\documentclass[11pt]{amsart}
\usepackage{geometry}                
\usepackage{fancyhdr}

\usepackage{amsmath}
\usepackage{mathrsfs}
\usepackage{graphicx}
\usepackage{amssymb}
\usepackage{epstopdf}
\usepackage[all]{xypic}
\usepackage{palatino}
\usepackage{float}
\usepackage{color}

\usepackage{hyperref}
\usepackage[arrow,curve,matrix]{xy}

\renewcommand{\det}{{\rm det}}

\renewcommand{\dim}{{\rm dim}}

\newcommand{\rk}{{\rm rk}}

\renewcommand{\det}{{\rm det}}

\theoremstyle{plain}

\newtheorem{thm}{Theorem}[section]

\newtheorem{cor}[thm]{Corollary}
\newtheorem{prop}[thm]{Proposition}
\newtheorem{lem}[thm]{Lemma}

\theoremstyle{definition}
\newtheorem{defn}[thm]{Definition}

\newtheorem{rmk}[thm]{Remark}

\def\ZZ{{\mathbb Z}}
\def\GG{{\textbf G}}

\def\PP{{\textbf P}}
\def\OO{\mathcal{O}}

\def\F{\mathcal{F}}
\def\P{\mathcal{P}}

\def\E{\mathcal{E}}

\def\I{\mathcal{I}}

\def\mm{\overline{\mathcal{M}}}

\newcommand{\bb}[1]{\mathbb{#1}}
\newcommand{\mc}[1]{\mathcal{#1}}

\newcommand{\defi}[1]{{\em #1}}
\newcommand{\op}[1]{\operatorname{#1}}

\newcommand{\bw}{\bigwedge}

\newcommand{\im}{\operatorname{Im}}

\renewcommand{\ker}{\operatorname{Ker}}

\newcommand{\oo}{\otimes}

\def\lra{\longrightarrow}

\newcommand{\Gr}{\operatorname{Gr}}

\newcommand{\Sym}{\operatorname{Sym}}

\theoremstyle{plain}

\newtheorem{lemma}[thm]{Lemma}
\newtheorem{proposition}[thm]{Proposition}

\theoremstyle{definition}
\newtheorem{remark}[thm]{Remark}
\newtheorem*{ack}{Acknowledgment}

\newtheorem*{thm-van*}{Vanishing Theorem}

\begin{document}
\title[Koszul modules with vanishing resonance in algebraic geometry]{Koszul modules with vanishing resonance in algebraic geometry}


\author[M. Aprodu]{Marian Aprodu}
\address{Marian Aprodu: Faculty of Mathematics and Computer Science, \hfill \newline\texttt{}
\indent	University of Bucharest,  Romania, and
\hfill \newline\texttt{}
\indent Simion Stoilow Institute of Mathematics
\hfill \newline\texttt{} \indent P.O. Box 1-764,
RO-014700 Bucharest, Romania}
\email{{\tt marian.aprodu@imar.ro}}

\author[G. Farkas]{Gavril Farkas}
\address{Gavril Farkas: Institut f\"ur Mathematik,   Humboldt-Universit\"at zu Berlin \hfill \newline\texttt{}
\indent Unter den Linden 6,
10099 Berlin, Germany}
\email{{\tt farkas@math.hu-berlin.de}}

\author[C. Raicu]{Claudiu Raicu}
\address{Claudiu Raicu: Department of Mathematics,
University of Notre Dame \hfill \newline\texttt{}
\indent 255 Hurley Notre Dame, IN 46556, USA, and \hfill\newline\texttt{}
\indent Simion Stoilow Institute of Mathematics, \hfill\newline\texttt{}
\indent  P.O. Box 1-764, RO-014700 Bucharest, Romania}
\email{{\tt craicu@nd.edu}}

\author[J. Weyman]{Jerzy Weyman}
\address{Jerzy Weyman: Institute of Mathematics,
Uniwersytet Jagiello\'nski \hfill \newline\texttt{}
\indent 30-348,
Krak\'ow, Poland}
\email{{\tt  jerzy.weyman@uj.edu.pl}}

\dedicatory{To the memory of \c Stefan Papadima}

\begin{abstract}
We discuss various applications of a uniform  vanishing result for the graded components of the finite length Koszul module associated to a subspace $K\subseteq \bigwedge^2 V$, where $V$ is a vector space.  Previously Koszul modules of finite length have been used to give a proof of Green's Conjecture on syzygies of generic canonical curves. We  now give applications to effective stabilization of cohomology of thickenings of algebraic varieties,  divisors on moduli spaces of curves, enumerative geometry of curves on $K3$ surfaces and to  skew-symmetric degeneracy loci. We also show that the instability of sufficiently positive rank $2$ vector bundles on curves is governed by resonance and give a splitting criterion.
\end{abstract}

\maketitle

\section{Introduction}

Given a suitably nice space (for instance a compact K\"ahler manifold) $X$, one can view its cohomology ring $H^{\bullet}(X, \mathbb C)$  as a module over its exterior algebra $E:=\bigwedge H^1(X,\mathbb C)$. Multiplication with a class $a\in H^1(X, \mathbb C)$ defines a complex on $H^{\bullet}(X,\mathbb C)$ and the jump loci for the cohomology of these complexes lead to the definition of the \emph{resonance variety} $\mathcal{R}(X)$ of $X$, which turned out to be instrumental in several investigations involving generic vanishing on varieties, see for instance \cite{Be}, \cite{DPS}, \cite{GL91}, \cite{LP}. This definition has then been extended by Suciu and Papadima \cite{PS-imrn} first to the case of finitely generated groups and then in \cite{PS15} to an entirely algebraic context. For important applications of these techniques to Torelli groups we refer to \cite{DP} and references therein. Closely related to the concept of resonance is the definition of a \emph{Koszul module}, initially introduced in \cite{PS-imrn} and \cite{PS15} to explain via the BGG correspondence homological properties of Alexander invariants (and more generally of quadratic algebras), then further studied in a purely algebraic context in \cite{AFPRW1} and \cite{AFPRW2}. We recall now this set-up.

\vskip 4pt

Suppose  $V$ is an $n$-dimensional complex vector space and let us fix a linear subspace $K\subseteq \bigwedge^2 V$. We denote by $K^\perp=(\bigwedge^2V/K)^\vee\subseteq \bigwedge^2V^\vee$ the orthogonal of $K$, viewed as
the space of skew-symmetric bilinear forms on $V$ vanishing on $K$.
Let $S:=\mbox{Sym}(V)$ be the polynomial algebra over $V$ and consider the Koszul complex
$$\cdots \longrightarrow \bigwedge^3 V\otimes S\stackrel{\delta_3}\longrightarrow \bigwedge^2 V\otimes S\stackrel{\delta_2}\longrightarrow V\otimes S\stackrel{\delta_1}\longrightarrow S\longrightarrow \mathbb C\longrightarrow 0.$$ According to \cite{PS15}, \cite{AFPRW1}, \cite{AFPRW2}, the {\em Koszul module} associated to $(V,K)$ is the graded $S$-module
\[W(V,K):=\mathrm{Coker}\Bigl\{\bigwedge^3V\otimes S\longrightarrow \Bigl(\bigwedge^2V/K\Bigr)\otimes S\Bigr\},\]
where the map in question is the projection $\bw^2 V \oo S \to (\bigwedge^2V/K)\otimes S$ composed with the Koszul differential $\delta_3$.
The grading is inherited from $\Sym V$ under the convention that $\bigwedge^2V/K$ is placed in degree ~$0$.
It is straightforward to see that the graded piece $W_q(V,K)$ of the Koszul module can be identified with the cohomology of the complex
$$K\otimes  \mbox{Sym}^q V \stackrel{\delta_{2,q}}\longrightarrow V\otimes \Sym^{q+1} V \stackrel{\delta_{1,q+1}}\longrightarrow \Sym^{q+2} V.
$$

\vskip 4pt

It is shown in \cite[Lemma 2.4]{PS15} that the support of the Koszul module $W(V,K)$ in the affine space $V^\vee$, if non-empty, coincides with the \emph{resonance variety}
$$\mc{R}(V,K):=\Bigl\{a\in V^\vee : \mbox{there exists }b\in V^\vee \mbox{ such that } a\wedge b\in K^\perp\setminus \{0\} \Bigr\}\cup \{0\}.$$
In particular, $W(V,K)$ has finite length if and only $\mc{R}(V,K)=\{0\}$. In \cite{AFPRW2} (see also \cite[Theorem 1.3]{AFPRW1}), we found an optimal characterization of those subpaces $K\subseteq \bigwedge ^2 V$ having trivial resonance and established the following equivalence:
\begin{equation}\label{equivalence}
\mathcal{R}(V, K)=\{0\}\Longleftrightarrow W_q(V,K)=0 \ \mbox{ for } q\geq \mbox{dim}(V)-3.
\end{equation}

We refer to Theorem \ref{thm:mainresult} for a  precise formulation of this result.
The paper \cite{AFPRW2} presents applications of the equivalence (\ref{equivalence}) to geometric group theory in the case $G$ is a finitely generated group, $V=H_1(G, \mathbb C)$  and $K^{\perp}=\mbox{Ker}\bigl\{\cup_G\colon \bigwedge^2 H^1(G, \mathbb C)\rightarrow H^2(G, \mathbb C)\bigr\}$. On the other hand, we explained in \cite{AFPRW1}  how by specializing to the tangent developable of a rational normal curve in
$\mathbb \PP^g$, one can prove \emph{Green's Conjecture}  \cite{G84} on syzygies of generic canonical curves of genus $g$ by applying the equivalence (\ref{equivalence}) to the case of the \emph{Weyman module}, which is a particular Koszul module corresponding to the choice $V=\Sym^{n-1}(U)$ and $K=\Sym^{2n-4}(U)$, with $U$ being a $2$-dimensional vector space. This has led to an alternate approach to Green's Conjecture (including an essentially optimal result in positive characteristic) different from the one of Voisin's \cite{V02}, \cite{V05}.

\vskip 4pt

This paper is devoted to the study of other important classes of Koszul modules with vanishing resonance that appear naturally in algebraic geometry. First, recalling that $V$ is an $n$-dimensional complex vector space, we note that if $\mathcal{R}(V, K)=\{0\}$ then $\mbox{dim}(K)\geq 2n-3$. We provide a refinement involving \emph{multiplicities} of the equivalence (\ref{equivalence}) in the case of $(2n-3)$-dimensional subspaces $K\subseteq \bigwedge^2 V$ as an equality of two particular divisor on the Grassmannian $\GG:=\Gr_{2n-3}\bigl(\bigwedge^2 V\bigr)$ parametrizing such subspaces, see Theorem \ref{thm=GreenToKoszul}. Denoting by $\mc{D}_{\mathfrak{Kosz}}$ the divisor consisting of subspaces $[K]\in \GG$ such that $W_{n-3}(V,K)\neq 0$ (with its natural scheme structure) and by $\mc{D}_{\mathfrak{Res}}$ the divisor consisting of those $[K]\in \GG$ with $\mathcal{R}(V,K)\neq \{0\}$, we have an equality of divisors
\begin{equation}\label{eq:divequal}
\mathcal{D}_{\mathfrak{Kosz}}=(n-2)\cdot \mathcal{D}_{\mathfrak{Res}}
\end{equation}
on the Grassmannian $\GG$. An immediate application of the equality (\ref{eq:divequal}) is then the calculation of what we call the \emph{resonance divisor} of a morphism of vector bundles
$$\phi \colon \bigwedge^2 \E \rightarrow \F,$$
where $\E$ and $\F$ are vector bundles on a stack $X$ with  $\mbox{rk}(\E)=e$ and $\mbox{rk}(\F)=2e-3$. We denote by $\mathfrak{Res}(\phi)$ the locus of points $x\in X$ such that the map $\phi(x)\colon \bigwedge^2 \E(x)\rightarrow \F(x)$ contains a pure tensor
$0\neq s_1\wedge s_2$ in its kernel. A parameter count quickly shows that when $\phi$ is sufficiently general, $\mathfrak{Res}(\phi)$ is a divisor on $X$.

\begin{thm}\label{thm:alt}
Given a morphism $\phi\colon \bigwedge^2 \E\rightarrow \F$ of vector bundles over $X$ with $\mathrm{rk}(\E)=e$ and $\mathrm{rk}(\F)=2e-3$, assuming $\mathfrak{Res}(\phi)$ is a divisor on $X$, its class is given by the formula
$$[\mathfrak{Res}(\phi)]=\frac{(2e-4)!}{(e-2)!\cdot (e-1)!}\Bigl(c_1(\F)-\frac{4e-6}{e}c_1(\E)\Bigr)\in CH^1(X).$$
\end{thm}

Theorem \ref{thm:alt} has numerous applications in moduli theory, one of them on the Kodaira dimension of the moduli space of Prym varieties having been presented in \cite{FJP}. While referring to Theorem \ref{thm:classVoisin} for  further applications to $K3$ surfaces, we discuss one consequence of Theorem \ref{thm:alt} to the geometry of the moduli space $\mm_{g,n}$ of $n$-pointed stable curves of genus $g$. For a smooth curve $C$, a \emph{canonical pencil} is the degree $2g-2$ cover $C\rightarrow \PP^1$ induced by two canonical forms without common zeroes. Since $C$ has a $(2g-4)$-dimensional family of canonical pencils each of them having finitely many ramification points, imposing the condition that $2g-3$ marked points are ramification points of such a pencil yields a divisorial condition in moduli.

\begin{thm}\label{thm:canpencil}
The class of the divisor $\mathfrak{Cp}_g$ of pointed curves $[C, x_1, \ldots, x_{2g-3}]\in \mm_{g,2g-3}$ such that $x_1, \ldots, x_{2g-3}$ are ramification points of a canonical pencil on $C$ is equal to
$$[\mathfrak{Cp}_g]=\frac{(2g-4)!}{(g-2)!\cdot (g-1)!}\Bigl(-\frac{2(2g-3)}{g}\lambda+3\sum_{i=1}^{2g-3}\psi_i\Bigr)\in CH^1(\mm_{g,2g-3}).$$
\end{thm}

Here $\lambda$ is the Hodge class, whereas $\psi_i$ denotes the cotangent class on $\mm_{g,2g-3}$ corresponding to the $i$-th marked point. Theorem
\ref{thm:canpencil} follows directly from Theorem \ref{thm:alt} by letting  $\E$ to be the Hodge bundle on $\mm_{g,2g-3}$, whereas $\F$ is the vector bundle having as fibre over a point $[C, x_1, \ldots, x_{2g-3}]\in \mm_{2g-3}$ the vector space
$H^0\bigl(C, \omega^{3}_{C|x_1+\cdots +x_{2g-3}}\bigr)$.

\vskip 4pt

\noindent{\bf Koszul modules associated to vector bundles.}
One can naturally associate a Koszul module to any vector bundle as we shall describe next. Suppose $E$ is a vector bundle on an algebraic
variety $X$ and consider the determinant map
\[
d\colon \bigwedge^2H^0(X,E)\to H^0(X,\bw^2 E).
\]
This gives rise to the following Koszul module
\begin{equation}\label{det-lekepezes}
W(X,E):=W(V,K),\mbox{ where } V:=H^0(X,E)^\vee \mbox{ and } K:=\ker(d)^\perp \subseteq \bigwedge^2V.
\end{equation}

If we let $\mc{R}(X,E) := \mc{R}(V,K)$ for $V$ and $K$ as above, then the non-triviality of the resonance amounts to the vector bundle $E$ carrying a \emph{subpencil}, that is, a line subbundle $L$ with $h^0(X,L)\geq 2$.
We show in \S \ref{subsec:detmaps}, that the equivalence (\ref{equivalence}) can be reformulated in this context as follows:
\begin{thm}\label{rk2}
Let $X$ be a projective variety with $H^1(X, \mathcal{O}_X)=0$ and let $E$ be a globally generated vector bundle on $X$. Then one has an isomorphism
\[W_q(X,E) \cong H^1\bigl(X,\Sym^{q+2}M_E\bigr)^\vee,\]
where $M_E$ denotes the kernel of the evaluation morphism $H^0(X,E)\otimes \mathcal O_X\rightarrow E$. If moreover $E$ has no subpencils, then it follows that $H^1\bigl(X, \mathrm{Sym}^{q+2} M_E\bigr)=0$, for $q\geq h^0(X,E)-3$.
\end{thm}

Theorem \ref{rk2} is particularly interesting for a polarized $K3$ surface $(X, L)$, where $L$ is an ample line bundle on $X$ with $L^2=2g-2$. Recall that the \emph{Mukai vector} of a sheaf $E$ on $X$ is defined as $v(E):=\bigl(\mbox{rk}(E), \mbox{det}(E), \chi(E)-\mbox{rk}(E)\bigr)\in H^{\bullet}(X)$ and that $M_L(v)$ denotes the moduli space of $L$-semistable sheaves on $X$ having Mukai vector $v$, see \S \ref{section:K3} for further details. A \emph{Lazarsfeld-Mukai} bundle is a globally generated vector bundle $E$ on $X$ with $H^1(X,E)=H^2(X,E)=0$.  If $E$ is a Lazarsfeld-Mukai bundle with Mukai vector $v(E)=(r,L,s)$, then $M_E^{\vee}$ is also a Lazarsfeld-Mukai bundle with vector $v(M_E^{\vee})=(s,L,r)$. Lazarsfeld-Mukai bundles have been instrumental in Voisin's proof of the Generic Green Conjecture \cite{V02}, \cite{V05}, Lazarsfeld's proof of the Petri Theorem \cite{L86}, or in the recent proof of the Mercat Conjecture \cite{BF}. In the case of $K3$ surfaces, Theorem \ref{rk2} implies the following result:

\begin{thm}\label{lm_eredmeny}
Let $X$ be a polarized $K3$ surface with $\mathrm{Pic}(X)=\mathbb Z\cdot L$ and let $E$ be a Lazarsfeld-Mukai bundle on $X$ with
$v(E)=(r,L,s)$. Then for all $b\geq r+s-1$ one has $$H^1\bigl(X,\mathrm{Sym}^b E\bigr)=0.$$
\end{thm}

In connection with Green's Conjecture, of particular relevance is the case $g=2r$ for $r\geq 2$, when $E$ is the \emph{unique} Lazarsfeld-Mukai bundle on $X$ having Mukai vector $v(E)=(r, L, 2)$. Theorem \ref{lm_eredmeny} reads in this case $$H^1\bigl(X,\Sym^{r+1}E\bigr)=0.$$  Remarkably, an independent geometric proof of the vanishing $H^1\bigl(X, \Sym^{r+1} E\bigr)=0$ (whose failure is a divisorial condition on the moduli space $\mathcal{F}_{g}$ of polarized $K3$ surfaces of genus $g$) would give yet another proof, different from Voisin's \cite{V02}, \cite{V05} or from that of \cite{AFPRW1} of the Generic Green Conjecture. Note that in this case
$$v\bigl(\Sym^{r+1} E\bigr)=\left({2r\choose r+1}, {2r\choose r}L, \frac{3r+2}{r}{2r\choose r+1}\right).$$

\vskip 4pt

\noindent{\bf The Voisin curve of a polarized $K3$ surface of odd genus.}

Assume now that $(X,L)$ is a polarized $K3$ surface of odd genus $g=2r+1\geq 11$, with $\mbox{Pic}(X)=\mathbb Z\cdot L$. The moduli space
$\widehat{X}:=M_L(2,L, r)$ turns out to be  a smooth $K3$ surface, called the \emph{Fourier-Mukai partner} of $X$. Furthermore, as explained in
\cite{M2}, there is a canonical way to endow $\widehat{X}$ with a  genus $g$ polarization $\widehat{h}$.

We fix a general curve $C\in |L|$, thus via Lazarsfeld's result \cite{L86}, the curve $C$ is Petri general of genus $2r+1$ and $W^1_{r+2}(C)$ is a smooth curve.  Voisin \cite{V1} associated to any pencil $A\in W^1_{r+2}(C)$ the Lazarsfeld-Mukai rank $2$ bundle $E_{C,A}\in \widehat{X}$, whose restriction $E_A:=E_{C,A}\otimes \mathcal{O}_C$ to $C$  sits in an extension
$$0\lra A\lra E_A\lra \omega_C\otimes A^{\vee}\lra 0,$$
such that  $h^0(C,E_A)=h^0(C,A)+h^0(C,\omega_C\otimes A^{\vee})=r+2$, see also  \cite{ABS1}. This assignment induces a map
$W^1_{r+2}(C)\rightarrow \widehat{X}$.  Since for a general $[E] \in \widehat{X}$, the restriction $E_A$ has canonical determinant and $h^0(C, \omega_C)=2h^0(X,E)-3$, we observe that the locus of vector bundles $[E]\in \widehat{X}$ whose restriction to $C$ has non-trivial resonance is a curve on $\widehat{X}$, which we call the \emph{Voisin curve} of the pair $(X,C)$. We have the following application of Theorem \ref{thm:alt} concerning the class of this curve:

\begin{thm}\label{thm:classVoisin}
Assume $\mathrm{Pic}(X)=\mathbb Z\cdot L$ and let $C\in |L|$ with $g(C)=2r+1$. The Voisin curve
$$\mathcal{R}(\widehat{X}, C):=\bigl\{[E]\in \widehat{X}: \mathcal{R}(C, E_C)\neq 0\bigr\}$$ is a curve in the linear system
$$\Bigl|\frac{(2r+1)!}{r!\cdot (r+2)!}\widehat{h}\Bigr|.$$ Furthermore, the map $W^1_{r+2}(C)\rightarrow \mathcal{R}(\widehat{X}, C)$ is a resolution of singularities of $\mathcal{R}(\widehat{X}, C)$.
\end{thm}

\vskip 3pt

\noindent{\bf Gaussian Koszul modules.}

Gaussian maps provide another context where Koszul modules appear naturally. Suppose $L$ is a very ample line bundle on a complex projective variety $X$ and denote by $\varphi_L\colon X\hookrightarrow \PP^r=\PP\bigl(H^0(X,L)^{\vee}\bigr)$ the corresponding embedding. Let $\I\subseteq \OO_{\PP^r}$ be the ideal sheaf of $X$. We then consider the Gaussian map
$$\psi_L:\bigwedge^2 H^0(X,L)\rightarrow H^0(X,\Omega_X^1\otimes L^2), \ \ \mbox{ } \ \mbox{  }  \psi_L(f\wedge g)=f dg-g df.$$
The cokernel of $\psi_L$ parametrizes deformations of the cone over the embedded variety $X\subseteq \PP^r$ inside $\PP^{r+1}$.
Wahl showed \cite{W} that for a curve $C$ lying on a $K3$ surface, the map $\psi_{\omega_C}$ is not surjective. A remarkable converse of this result has been recently established by Arbarello-Bruno-Sernesi \cite{ABS2}.

\vskip 3pt
We fix $(X,L)$ and set  $V:=H^0(X,L)^{\vee}$ and $K^{\perp}:=\mbox{Ker}(\psi_L)$, to obtain a Koszul module
$$\mathcal{G}(X,L):=W(V,K)$$
whose resonance is always trivial. We have the following result, indirectly  concerning the Koszul module $\mathcal{G}(X,L)$:

\begin{thm}\label{infneigh}
Let $X\subseteq \PP^r$ be a smooth variety satisfying $q(X)=0$ and $H^0\bigl(X,\Omega_X^1(1)\bigr)=0$. If $X_b\subseteq \PP^r$ is the $b$-th infinitesimal neighborhood of $X$ defined by the ideal $\I^{b+1}$, then the maps
$$H^0\bigl(X_b,\OO_{X_b}(a)\bigr)\rightarrow H^0\bigl(X_{b-1}, \OO_{X_{b-1}}(a)\bigr)$$
are isomorphisms for all $b\geq a\geq r$.
\end{thm}

A more general version of Theorem \ref{infneigh}, without any assumptions on $X$, is provided by Theorem \ref{thm:koszul-Wahl}. To place this result  into context, we recall that Hartshorne \cite{hartshorne2} showed that for a vector bundle $\mathcal{F}$ on $\PP^r$ and a closed subvariety $X\subseteq \PP^r$, for all $j\geq 0$ the
maps $H^j\bigl(X_b, \mathcal{F}\otimes \OO_{X_b}\bigr)\rightarrow H^j\bigl(X_{b-1}, \mathcal{F}\otimes \OO_{X_{b-1}}\bigr)$ are isomorphisms if $b\gg 0$. A quantitative version of these results for $\mathcal{F}=\OO_{\PP^r}(a)$ has been recently obtained in \cite{bhatt}: The restriction maps $$H^j\bigl(X_b,\OO_{X_b}(a)\bigr)\rightarrow H^j\bigl(X_{b-1}, \OO_{X_{b-1}}(a)\bigr)$$ are isomorphisms for all $j\geq 0$ as long as $b\geq \mbox{dim}(X)+a+1$, see \cite[Remark 2.18]{bhatt}. Our Theorem \ref{infneigh} can be viewed as a significant improvement (under certain assumptions) of this result at the level of global sections.

\vskip 3pt

Concerning the hypothesis of Theorem \ref{infneigh}, they are satisfied for most Fano varieties (for instance for all Hermitian Symmetric Spaces of type A, B, C or D, see \cite{Sn}). Also, if  $X$ is a Fano threefold then always $q(X)=0$, whereas from the Iskosvskikh-Mukai classification it follows that the condition $H^0(X,\Omega_X^1(1))\neq 0$ implies that $X$ is of index one and has genus $10$ or $12$, see \cite{JR}. The hypothesis of Theorem \ref{infneigh} are also satisfied for many varieties of Kodaira dimension zero. For instance, if $(X,L)$ is a polarized $K3$ surface of degree $L^2=2g-2$, the condition
$H^0\bigl(X, \Omega_X^1\otimes L\bigr)=0$ is equivalent to the statement that a general curve of genus $g$ lies on a $K3$ surface and is thus satisfied if and only $g\leq 9$ or $g=11$, see \cite{Be2}.

\begin{ack}
Above all, we acknowledge the important contribution of \c Stefan Papadima. This paper, which is a natural continuation of \cite{AFPRW1} and \cite{AFPRW2} is part of a project that was initiated by him. We also profited from numerous discussions with Alex Suciu  related to this circle of ideas.

\vskip 3pt

{\small{Aprodu  was supported by the Romanian Ministry of Research and
Innovation, CNCS - UEFISCDI, grant
PN-III-P4-ID-PCE-2020-0029, within PNCDI III.  Farkas was supported by the DFG Grant \emph{Syzygien und Moduli} and by the ERC Advanced Grant SYZYGY.
Raicu was supported by the NSF Grants No.~1901886 and ~2302341. Weyman was supported by the grants MAESTRO NCN - UMO-2019/34/A/ST1/00263 - Research in Commutative Algebra and Representation Theory and NAWA POWROTY - PPN/PPO/2018/1/00013/U/00001 - Applications of Lie algebras to Commutative Algebra. This project has received funding from the European Research Council (ERC) under the European Union Horizon 2020 research and innovation program (grant agreement No. 834172)
}}
\end{ack}

\section{Basics on Koszul modules}
\label{sec:kmod}

We recall the basic definitions of Koszul modules following  \cite{AFPRW1}, \cite{AFPRW2}, \cite{PS15}. For simplicity, we stick to characteristic zero and let $V$ be a complex vector space of dimension $n\ge 2$ and denote by $S:=\Sym V$ the symmetric algebra of~$V$. We consider the standard grading on $S$, where the elements in $V$ are of degree 1. We fix a linear subspace $K\subseteq \bigwedge^2V$ of dimension~$m$ and denote by $\iota\colon K\to \bw^2V$ the inclusion and let $K^\perp:=\ker(\iota^{\vee})\subseteq \bigwedge^2 V^{\vee}$. We introduce the Koszul differentials
\[
\delta_p\colon \bigwedge^pV\otimes S\to \bigwedge^{p-1}V\otimes S,
\]
\[
\delta_p(v_1 \wedge \cdots \wedge v_p \oo f) = \sum_{j=1}^p (-1)^{j-1} v_1 \wedge \cdots \wedge \widehat{v_j} \wedge \cdots \wedge v_p \oo v_j f.
\]
We have a decomposition $\delta_p=\bigoplus_q\delta_{p,q}$ into graded pieces, where
\begin{equation}\label{eq:delta-pq}
\delta_{p,q}:\bw^p V \oo \Sym^q V \to \bw^{p-1}V \oo \Sym^{q+1}V.
\end{equation}

\medskip
The Koszul module $W(V,K)$ defined in the Introduction is a graded $S$-module, whose degree $q$ component has the following description.
\begin{equation}
\label{eqnWqDef}
W_q(V,K)=\mathrm{Coker}\Bigl\{\bigwedge^3V\otimes \Sym^{q-1}V\lra \Bigl(\bigwedge^2V/K\Bigr)\otimes \Sym^qV\Bigr\}.
\end{equation}

Since the Koszul complex is exact, it is often convenient to realize $W_q(V,K)$ as the middle cohomology of the following complex of vector spaces:
\begin{equation}\label{eqn:W}
\xymatrixcolsep{5pc}
\xymatrix{
K \oo \Sym^q V \ar[r]^{\delta_{2, q}|_{K \oo \Sym^q V}} & V\oo \Sym^{q+1} V \ar[r]^{\delta_{1,q+1}} & \Sym^{q+2} V.
}
\end{equation}

As pointed out in \cite{PS15} and further explained in \cite{AFPRW1}, \cite{AFPRW2}, the construction of Koszul modules displays good functoriality properties.  For instance, if  $K \subseteq K'\subseteq \bigwedge^2 V$ are linear susbspaces, one has an
induced  surjective morphism of graded $S$-modules
\begin{equation}
\label{eq:wnat}
W(V,K) \twoheadrightarrow W(V,K').
\end{equation}

\subsection{Resonance varieties}
\label{subsec:resonance}

Building on work of Green-Lazarsfeld \cite{GL91}, Dimca-Papadima-Suciu \cite{DPS} and others,  Papadima and Suciu \cite{PS15} gave an algebraic definition of the resonance variety associated to a pair $(V,K)$ as above, which we now recall.

\begin{defn}\label{def:resonance}
The {\em resonance variety} associated  to the pair $(V,K)$ is the locus
\begin{equation}
\label{eq:defr}
\mathcal R(V,K):=\left\{a\in V^\vee : \mbox{there exists }b\in V^\vee \mbox{ such that } a\wedge b\in K^\perp\setminus \{0\} \right\}\cup \{0\}
\end{equation}
\end{defn}

The resonance variety $\mathcal{R}(V,K)$ is the union of $2$-dimensional subspaces of $V^{\vee}$ parameterized by the intersection $\PP K^\perp \cap\mathrm{Gr}_2(V^\vee )$, where $\mathrm{Gr}_2(V^{\vee})\subseteq \PP\bigl(\bigwedge^2 V^{\vee}\bigr)$ is the Pl\"ucker embedding. Setting up the diagram
\[
\xymatrixcolsep{5pc}
\xymatrix{
\Xi \ar[r]^{\pi} \ar[d]^{p_1}& \mathrm{Gr}_2(V^\vee )\\
\mathbf{P} V^\vee &
}
\]
where $\Xi\subseteq \PP V^\vee\times \mathrm{Gr}_2(V^\vee )$ is the incidence variety, we observe that $\mathcal{R}(V,K)$ is the affine cone over the following projective variety $$\mathbf{R}(V,K) :=p_1\bigl(\pi^{-1}(\PP K^\perp\cap\mathrm{Gr}_2(V^\vee ))\bigr),$$ which we refer to as the \emph{projectivized resonance variety} of $(V,K)$. Note that the correspondence $p_1\circ \pi^{-1}$ mapping a point $[a\wedge b]\in  \PP K^\perp\cap\mathrm{Gr}_2(V^\vee )$ to the line $\ell_{ab}$ in $\mathbf{P}V^\vee$ passing through $[a]$ and $[b]$ gives a natural bijection between $\PP K^\perp \cap\mathrm{Gr}_2(V^\vee )$ and the set of lines contained in $\mathbf{R}(V,K)$, whose inverse is $\ell_{ab}\mapsto\pi (p_1^{-1}([a])\cap\ p_1^{-1}([b]))$.

\vskip 3pt

It was showed in \cite[Lemma~2.4]{PS15} that away from $0$, the support of the graded $S$-module $W(V, K)$ inside $V^\vee $ coincides with the resonance variety $\mathcal{R}(V, K)$.  In particular,
\begin{equation}
\label{eq:mainequiv}
\PP K^\perp\cap\mathrm{Gr}_2(V^\vee )=\emptyset \Longleftrightarrow \mathcal R(V,K)= \{0\}
\Longleftrightarrow \dim_{\mathbb C}W(V,K) < \infty.
\end{equation}

In \cite{AFPRW1} we provide a sharp vanishing result for Koszul modules with vanishing resonance. This is the starting point for many of the geometric applications in this paper.

\begin{thm}\label{thm:mainresult}
Let $V$ be a complex $n$-dimensional vector space and let $K\subseteq\bw^2 V$ be a subspace such that $\mc{R}(V,K)=\{0\}$. We have that $W_{q}(V,K)=0$ for all $q\geq n-3$.  Furthermore, if $\mathcal R(V,K)=\{0\}$, then the following inequality holds
\[\dim\, W_q(V,K) \leq {n+q-1 \choose q} \frac{(n-2)(n-q-3)}{q+2}, \mbox{ for  }q=0, \ldots, n-4,\]
with equality if $\mbox{dim}(K)=2n-3$.
\end{thm}

The connection between resonance and Koszul modules shows that the resonance carries a natural \emph{scheme} structure which, in some cases might be non--reduced. In the forthcoming paper \cite{AFRS} we shall have a close look at this phenomenon.

\subsection{Isotropy and separability}(see \cite{AFRS})
In geometric situations (like those when the resonance variety parametrizes complexes with jumping cohomology in the spirit of \cite{DPS}), the resonance variety $\mathcal{R}(V,K)$ often enjoys further properties, which we summarize in a definition. Before formulating it, let $E:=\bw V^{\vee}$ be the exterior algebra on the vector space
$V^{\vee}$, and write $\langle U \rangle_E$ for the ideal in $E$ generated
by a subset $U\subseteq E$.

\begin{defn}\label{def:isotropy}
We say that a subspace $\overline{V}^\vee\subset V^\vee$ is
\begin{itemize}
\item \defi{Isotropic}, if
 $\bw^2 \overline{V}^{\vee} \subseteq K^{\perp}$.
\item \defi{Separable} if $K^{\perp} \cap \langle \overline{V}^{\vee}\rangle_E \subseteq \bw^2 \overline{V}^{\vee}$.
\item \defi{Strongly isotropic}
if it is separable and isotropic, that is, if
$K^{\perp} \cap \langle \overline{V}^{\vee}\rangle_E = \bw^2 \overline{V}^{\vee}$.
\end{itemize}
\end{defn}

Similar definitions can be given for the projective subspaces of $\PP V^\vee$.

\begin{defn}\label{def:resprop}
We say that the resonance variety $\mathcal{R}(V,K)$ is
\begin{itemize}
\item \defi{Linear}, if $\mathcal{R}(V,K)$ is a union of linear subspaces of $V^{\vee}$, that is,
$$\mathcal{R}(V,K)=\overline{V}_1\cup \cdots \cup \overline{V}_s.$$
\item \defi{Isotropic}, \defi{separable}, or \defi{strongly isotropic} if it is linear and each component $\overline{V}_i^{\vee}$ of $\mathcal{R}(V,K)$  is isotropic, separable,  or strongly isotropic, respectively.
\end{itemize}
\end{defn}

\medskip

For the relevance of these conditions in the case of resonance varieties associated to hyperplane arrangements we refer to \cite{CSch}. In the paper \cite{AFRS} we relate separability to the reduceness of the projectived resonance scheme and establish an optimal effective version of Chen's rank conjecture for Koszul modules with strongly isotropic resonance.

Note that if two lines contained in $\mathbf{R}(V,K)$ intersect, then the whole plane they generate is contained in $\mathbf{R}(V,K)$. If $[a]\in \mathbf{R}(V,K)$, then the projectivization of the subspace
\[
\mathcal{F}(a):=\bigl\{b\in V^\vee: a\wedge b \in K^\perp\bigr\}
\]
is contained in $\mathbf{R}(V,K)$ and is the maximal projective subspace inside $\mathbf{R}(V,K)$ that passes through $[a]$. Moreover the set $\bigl\{[a\wedge b]: b\in \mathcal{F}(a)\bigr\}$ is contained in $\mathbf{G}\cap\mathbf{P}K^\perp$.

\begin{lemma}
The map $\PP V^{\vee}\ni [a]\mapsto \mathrm{dim}(\mathcal{F}(a))$ is upper--semicontinous.
\end{lemma}

\proof
Consider $\mathcal{F}$ the kernel of the composed sheaf morphism on $\PP:=\PP V^{\vee}$
\[
V^\vee\otimes \mathcal{O}_\mathbf{P}(-1)\longrightarrow  \bigwedge^2V^\vee\otimes\mathcal{O}_\mathbf{P}\longrightarrow \bigwedge^2V^\vee/K^\perp\otimes\mathcal{O}_\mathbf{P}.
\]
Then $\mathcal{F}(a)$ as defined above can be identified with the fibre of $\mathcal{F}$ at $[a]$ and we apply Grauert's Theorem.
\endproof

\begin{proposition}
\label{prop:isolated}
If  $[a\wedge b]$ is an isolated point of $\PP K^\perp\cap\mathrm{Gr}_2(V^\vee )$, then the line $\ell_{ab}$ joining $[a]$ and $[b]$ is an isotropic connected component of $\mathbf{R}(V,K)$.
\end{proposition}

\proof
We prove first that the line $\ell_{ab}\subseteq \PP V^{\vee}$ is a connected component of $\mathbf{R}(V,K)$. Since $[a\wedge b]$ is an isolated point, and the projectivization of $\bigl\{a\wedge b': b'\in \mathcal{F}(a)\bigr\}$ would be contained in $\PP K^\perp\cap \mathrm{Gr}_2(V^\vee )$, it follows that $\mathcal{F}(a)$ is $2$-dimensional, spanned by $a$ and $b$. Denote this subspace by $\overline{V}^\vee$. By semicontinuity, for each $a'$ in a neighborhood of $[a]$ in $\mathbf{R}(V,K)$, we have $\mathrm{dim}(\mathcal{F}(a'))=2$. If $\ell_{ab}$ is not a connected component, then we have a sequence $([a_n])_n\in \mathbf{R}(V,K)\setminus\ell_{ab}$ converging to $[a]$. Without loss of generality, we assume $a_n\rightarrow a$ in $\mathcal{R}(V,K)$ and $\mathrm{dim}(\mathcal{F}(a_n))=2$ for all $n$. Hence we obtain a sequence of lines $(\ell_n)_n$ contained in $\mathbf{R}(V,K)$ and different from $\ell_{ab}$, converging to the limit $\ell_{ab}$. This corresponds to a sequence of points in $\PP K^\perp\cap\mathrm{Gr}_2(V^\vee )$ converging to $[a\wedge b]$, which is impossible. The isotropy of $\overline{V}^\vee$ is straightforward. Indeed, it is isotropic if and only if $a\wedge b\in K^\perp$, which is true by hypothesis.
\endproof

Proposition \ref{prop:genus8} provides one application of Proposition \ref{prop:isolated} in geometric setting.

\section{The Chow form of the Grassmannian of lines and alternating degeneracy loci}
\label{subsec:divisors}

We begin by recording a well-known sufficient conditions for the supports of two Cartier divisors on an algebraic variety to be equal.
Let $X$ be a smooth quasi-projective variety and $\mathcal A$ and $\mathcal B$  vector bundles on $X$ of the same rank $r$ and let
$\varphi \colon \mathcal A\to \mathcal B$ be a vector bundle morphism. Assume its degeneracy locus
\[
D(\varphi):=\Bigl\{x\in X: \mathrm{rk } \bigl\{\varphi(x): \mathcal A(x)\longrightarrow \mathcal B(x)\bigr\} \le r-1\Bigr\}.
\]
is a Cartier divisor on $X$, that is,  $D(\varphi)\neq X$
and that for any point $x$ in an irreducible component $Z$ of $D(\varphi)$, we have $\mathrm{dim} \ \ker(\varphi(x)) \geq k$.
Then $Z$ enters with multiplicity at least $k$ in $D(\varphi)$. We shall use the following well-known fact,  presented here for the convenience of the reader.

\begin{lem}
\label{lem:divisorial}
Let $Y$ be an irreducible projective variety and $U\subseteq Y$ an open subset with $\mathrm{codim}(Y\setminus U, Y)\ge 2$.
Assume $\mathcal A$ and $\mathcal B$ are vector bundles of the same rank on $U$, and we are given a morphism $\varphi\colon \mathcal A\to\mathcal B$, whose degeneracy locus $D_1=D(\varphi)$ is a genuine divisor. Let $D_2$ be a reduced Cartier divisor on $U$ such that
$\mathrm{Supp}(D_2)\subseteq\mathrm{Supp}(D_1)$ and  $[D_1]=k[D_2]\in CH^1(U)$ for some positive integer $k$.
If for any $x\in\mathrm{Supp}(D_2)$ we have $\mathrm{dim } \ker(\varphi(x)) \geq k$, then $\mathrm{Supp}(D_1) = \mathrm{Supp}(D_2)$ and $D_1=k\cdot D_2$ as divisors.
\end{lem}

\proof
The hypotheses imply that the only effective divisor $D$ on $U$ whose rational class is zero is the zero--divisor itself. Indeed, if $D\neq 0$, then its  closure $\overline{D}$ in $Y$ satisfies $\overline{D}\cdot H^{\mathrm{dim}(Y)-1}>0$ for any ample divisor $H$ on $Y$, a contradiction. We apply this  to the divisor $D:=D_1-k D_2$, which is  effective, for, as explained,  $D_2$ enters in $D_1$ with multiplicity at least $k$.
\endproof

Throughout this section let $V$ be an $n$-dimensional complex vector space and set $\GG:=\Gr_{2n-3}\bigl(\bigwedge^2 V\bigr)$. Theorem \ref{thm:mainresult} offers a set-theoretic description of the \emph{Koszul divisor}
\[
\mathcal{D}_\mathfrak{Kosz}:=\Bigl\{K\in\GG: W_{n-3}(V,K)\ne 0\Bigr\}.
\]
The fact that $\mathcal{D}_{\mathfrak{Kosz}}$ is a divisor on $\GG$ follows once we observe that  if $\mathcal U$ is the universal rank--$(2n-3)$ subbundle on $\GG$, then $\mathcal D_\mathfrak{Kosz}$
is the degeneracy locus of the morphism

\begin{equation}\label{eq:universal-delta2n-3}
\gamma\colon \mathcal{U}\otimes \Sym^{n-3}V\to \mathcal O_{\GG}\otimes \mathrm{Im}(\delta_{2,n-3}),
\end{equation}
which in the fiber over a point $[K]\in \GG$ is given by the Koszul differential $\delta_{2,n-3}$. Theorem \ref{thm:mainresult} implies that $\gamma$ is non-degenerate; for instance if we write $V=\Sym^{n-1}(U)$, with $U$ being a $2$-dimensional vector space, then we have established in \cite{AFPRW2} that the point $$\Bigl[K:=\Sym^{2n-4} U\subseteq \bigwedge^2 \Sym^{n-1} U\Bigr]\notin \mathcal{D}_\mathfrak{Kosz},$$ and therefore $\gamma$ is non-degenerate and $\mathcal{D}_\mathfrak{Kosz}$ is a genuine divisor on $\GG$.

\vskip 3pt

On the other hand, we can consider the \emph{Cayley--Chow form} of the Grassmannian
$\mathrm{Gr}_2(V^\vee )\subseteq  \PP\bigl(\bigwedge^2 V^\vee \bigr)$. Explicitly, this divisor is the locus
\[
\mathcal{D}_\mathfrak{Res}:=\Bigl\{K\in \GG: \PP(K^\perp)\cap \mathrm{Gr}_2(V^\vee )\ne\emptyset\Bigr\}
\]
and comes with an induced scheme structure. Theorem \ref{thm:mainresult} (see Theorem 1.3 from \cite{AFPRW1} for a version in positive characteristic) can then be formulated as a set-theoretic equality:
\begin{equation}\label{egyenloseg}
\mathrm{Supp}(\mathcal{D}_\mathfrak{Res})=\mathrm{Supp}(\mathcal{D}_\mathfrak{Kosz}).
\end{equation}

The divisor classes of $\mathcal{D}_\mathfrak{Res}$ and $\mathcal{D}_\mathfrak{Kosz}$
are easy to describe in terms of the generator $\mathcal L = \det(\mc{U}^{\vee})$
of the Picard group  $\mathrm{Pic}(\GG)$, which is the hyperplane section bundle coming from the Pl\"ucker embedding of $\GG$.
It follows from (\ref{eq:universal-delta2n-3}) that the degree of $\mathcal{D}_{\mathfrak{Kosz}}$ equals the dimension of $\Sym^{n-3}(V)$, which proves that the divisor class of $\mathcal{D}_\mathfrak{Kosz}$ equals

\begin{equation}
\label{eq:class-D1}
[\mathcal{D}_\mathfrak{Kosz}]={2n-4 \choose n-1}[\mathcal{L}].
\end{equation}

To compute the class of $\mathcal{D}_\mathfrak{Res}$ we recall that the degree of a Cayley--Chow form equals the degree of the variety to which it is associated \cite[Corollary 2.1]{DS}, which in our case is equal to the \emph{Catalan number} $C_{n-2}=\frac{1}{n-1}{2n-4 \choose n-2}$, see  \cite[Proposition~4.12]{EH16}. Hence, we have that the divisor classes of $\mathcal{D}_\mathfrak{Kosz}$ and $\mathcal{D}_\mathfrak{Res}$ are related by
\begin{equation}
\label{eq:class-D2}
[\mathcal{D}_\mathfrak{Kosz}]=(n-2)[\mathcal{D}_\mathfrak{Res}].
\end{equation}

\begin{lem}
\label{lem:incl}
We have a set-theoretical inclusion $\mathrm{Supp}(\mathcal{D}_\mathfrak{Res})\subseteq\mathrm{Supp}(\mathcal{D}_\mathfrak{Kosz})$.
\end{lem}

\proof
Let $[K]\in \mathcal D_\mathfrak{Res}$. By Theorem \ref{thm:mainresult}, the Koszul module $W(V,K)$ is of infinite length.
Since it is generated in degree zero, it follows that $W_q(V,K)\ne 0$ for all $q\ge 0$, and in particular $W_{n-3}(V,K)$ is also non-zero.
\endproof

\begin{lem}
\label{lem:graphic}
For any $[K]\in \mathcal{D}_\mathfrak{Res}$ we have $\mathrm{dim}\ W_q(V,K)\ge q+1$, for all $q\ge 0$.
\end{lem}

\proof
For $[K]\in \mathcal{D}_\mathfrak{Res}$, it follows from \eqref{eq:defr} and \eqref{eq:mainequiv} that we may find a basis $\{v_1,\ldots,v_n\}$ of $V$ such that $v_1^\vee \wedge v_2^\vee \in K^\perp$. We get that $K\subseteq K'$, where $K'\subseteq \bigwedge ^2V$ is the codimension one subspace with basis $v_i\wedge v_j$ with $1\le i<j\le n$ and $(i,j)\ne (1,2)$. A direct calculation shows that the Hilbert series of $W(V,K')$ equals $\sum_{q\ge 0}(q+1)t^q$, while \eqref{eq:wnat} proves that the graded module $W(V,K')$ is a quotient of $W(V,K)$, concluding our proof.
\endproof

\vskip 3pt

The following result is a refinement of Theorem \ref{thm:mainresult} and provides an explicit description, including multiplicities, of the Chow form of the Grassmannian $\Gr_2(V^{\vee})$ in its Pl\"ucker embedding.

\begin{thm}\label{thm=GreenToKoszul}
One has the following equality of divisors on $\GG$
$$\mathcal{D}_\mathfrak{Kosz}=(n-2)\cdot \mathcal{D}_\mathfrak{Res}.$$
\end{thm}

\proof
If $n=3$ then $m=3$ and therefore $K=\bw^2 V$, which implies that $W(V,K)=0$. Assume from now on $n\geq 4$ and
we take $D_1:=\mc{D}_\mathfrak{Kosz}$ and $D_2:=\mc{D}_\mathfrak{Res}$, for which we apply Lemma~\ref{lem:divisorial}: we know by Lemma~\ref{lem:incl} that $\mathrm{Supp}(D_2)\subseteq\mathrm{Supp}(D_1)$, while (\ref{eq:class-D2}) shows that $[D_1] = (n-2)\cdot[D_2]$; by Lemma~\ref{lem:graphic} with $q=n-3$, it follows that over the point $[K]\in\textrm{Supp}(D_2)$ the fiber of the map (\ref{eq:universal-delta2n-3}) has cokernel $W_{n-3}(V,K)$ of dimension at least $q+1=n-2$, so Lemma~\ref{lem:incl} applies with $k=n-2$ showing the equality of divisors $D_1=(n-2)D_2$, as desired.
\endproof

\begin{rmk}
One has two remarkable equalities of divisors, namely  $\mc{D}_{\mathfrak{Kosz}}=(n-2)\cdot\mc{D}_{\mathfrak{Res}}$ on $\GG$, respectively the equality
$\mathfrak{Syz}=(n-2)\cdot \mathcal{M}_{2n-3,n-1}^1$ on the moduli space $\mathcal{M}_{2n-3}$ of curves of genus $2n-3$, where $\mathcal{M}_{2n-3,n-1}^1$ is the $(n-1)$-gonal locus, whereas
$$\mathfrak{Syz}:=\bigl\{[C]\in \mathcal{M}_{2n-3}: K_{n-2,1}(C,\omega_C)\neq 0\bigr\}$$
is the locus of curves with a non-trivial $(n-2)$nd syzygy in their canonical embedding.

It would be highly interesting to establish a  direct geometric connection between these equalities and also explain the occurrence of the same multiplicity $n-2$.
\end{rmk}

\subsection{The resonance divisor of a skew-symmetric degeneracy locus.} We present now an application of Theorem \ref{thm=GreenToKoszul} to a situation appearing frequently in moduli theory. Assume we are given two vector bundles $\E$ and $\F$ over a stack $X$ such that $\mbox{rk}(\E)=e$ and $\mbox{rk}(\F)=2e-3$ where $e\geq 3$, and a generically surjective morphism of vector bundles
$$\phi\colon \bigwedge^2 \E\rightarrow \F.$$

Identifying the Grassmannian $\Gr_2\bigl(\E(x)\bigr)\subseteq \PP\bigl(\bigwedge^2 \E(x)\bigr)$ of lines in the fibre $\E(x)$ over a point $x\in X$ with the (projectivization of the) space of rank  $2$ exterior tensors on $\E(x)$, the numerical conditions at hand imply that the locus
$$\mathfrak{Res}(\phi):=\Bigl\{x\in X: \exists \ 0\neq s_1\wedge s_2\in \bigwedge^2 \E(x): \phi(s_1\wedge s_2)=0\Bigr\}$$
is a virtual divisor on $X$. We assign a divisor structure to this locus as follows.

Let $\Sigma$ be the variety consisting of pairs $(\varphi, K)$, where $\varphi\in \mbox{Hom}\bigl(\bigwedge ^2 \mathbb C^e, \mathbb C^{2e-3}\bigr)$, and $K\subseteq \mbox{Ker}(\varphi)$ is a subspace of codimension $2e-3$. For a morphism of vector bundles $\phi\colon \bigwedge^2 \E\rightarrow \F$ as above, over a trivializing open set $U\subseteq X$ consider the fibre product
$$\Sigma(\phi):=U\times_{\mathrm{Hom}\bigl(\bigwedge^2 \mathbb C^e, \mathbb C^{2e-3}\bigr)} \Sigma$$ endowed with the projections $\pi_1\colon \Sigma(\phi)\rightarrow X$ and $\pi_2\colon \Sigma(\phi)\rightarrow \Gr_{2e-3}\bigl(\bigwedge^2 (\mathbb C^e)^{\vee}\bigr)$.

\vskip 3pt

\begin{defn}\label{def:resdiv}
We define the virtual divisor  $\mathfrak{Res}(\phi)$ of the morphism $\phi\colon \bigwedge^2 \E\rightarrow \F$ locally over a trivializing open set $U$ as $\mathfrak{Res}(\phi)_{|U}:=(\pi_1)_*\bigl(\pi_2^*\mathcal{D}_{\mathfrak{Res}}\bigr)=
\frac{1}{e-2}(\pi_1)_*\bigl(\pi_2^*\mathcal{D}_{\mathfrak{Kosz}}\bigr)$.
\end{defn}

We can now prove Theorem \ref{thm:alt}, which provides a formula for the class of this locus in terms of the first Chern classes of $\E$ and $\F$:

\vskip 3pt

\emph{Proof of Theorem \ref{thm:alt}.}
We may assume $e\geq 4$ and consider the chain of morphisms
$$\E\otimes \Sym^{e-2}(\E)/\Sym^{e-1}(\E)\stackrel{\delta_{2,e-3}^{\vee}}\longrightarrow \bigwedge^2 \E\otimes \Sym^{e-3}(\E)\stackrel{\phi\otimes \mathrm{id}}\longrightarrow \F\otimes \Sym^{e-3}(\E),$$
and denote by $\vartheta\colon \E\otimes \Sym^{e-2}(\E)/\Sym^{e-1}(\E)\rightarrow \F\otimes \Sym^{n-3}(\E)$ the composition. Applying Theorems \ref{thm:mainresult} and \ref{thm=GreenToKoszul}, we infer that $(e-2)\cdot \mathfrak{Res}(\phi)$ is equal as a divisor to the degeneracy locus of the morphism $\vartheta$. Using the formula $c_1(\Sym^n \E)={e+n-1 \choose e} c_1(\E)$ valid for all $n\geq 0$, we compute
$$(e-2)\cdot[\mathfrak{Res}(\phi)]=c_1\bigl(\F\otimes \Sym^{e-3}(\E)\bigr)-c_1\bigl(\E\otimes \Sym^{e-2}(\E))+c_1\bigl(\Sym^{e-1}\E)=$$
$${2e-4\choose e-3}c_1(\F)+
(2e-3){2e-4\choose e-4}c_1(\E)-{2e-3\choose e-2} c_1(\E)-e{2e-3\choose e-3}c_1(\E)$$
$$+{2e-2\choose e-2}c_1(\E)={2e-4\choose e-3}\Bigl(c_1(\F)-\frac{4e-6}{e}c_1(\E)\Bigr),$$
which immediately leads to the claimed formula.
\hfill
$\Box$

\section{Koszul modules associated to vector bundles}
\label{subsec:detmaps}

We now discuss a class of Koszul modules naturally associated to vector bundles.
For a vector bundle $E$ on a projective variety $X$, we consider the determinant map
$$
d\colon \bigwedge^2 H^0(X,E)\to H^0\bigl(X,\bw^2 E\bigr)
.$$

\begin{defn}\label{ex:W(X,E)}
The Koszul module associated to the pair $(X, E)$ as above is defined as
\[
W(X,E):=W(V,K),\mbox{ where }V:=H^0(X,E)^\vee \mbox{ and } K=\ker(d)^\perp=\mathrm{Im}(d)^\vee\subseteq \bigwedge^2V.
\]
\end{defn}

The triviality of the resonance variety $\mathcal{R}(X,E)$ associated to the Koszul module $W(X, E)$ has a transparent geometric interpretation.
\begin{prop}
One has  $\mathcal R(X,E)=\{0\}$ if and only if $E$ has no locally free subsheaf of rank one $L$ with $h^0(X,L)\ge 2$.
\end{prop}
\begin{proof}
Indeed, via (\ref{eq:mainequiv}), the resonance $\mathcal R(X,E)$ is non-trivial if and only if we can find sections $s_1, s_2\in H^0(X,E)$ with $0\neq s_1\wedge s_2\in K^{\perp} = \ker(d)$, which in turn is equivalent to the fact that $s_1$ and $s_2$ generate a rank-one subsheaf whose double dual is a locally free subsheaf of $E$.
\end{proof}

If the vector bundle $E$ in Definition \ref{ex:W(X,E)} is globally generated, then the corresponding Koszul module can be given a geometric description in terms of kernel bundles:

\begin{thm}
\label{thm:det}
Let $X$ be a projective variety and let $E$ be a globally generated vector bundle on $X$ such that the determinant map
\[
d \colon \bigwedge^2H^0(X,E)\to H^0\bigl(X,\bigwedge^2E\bigr)
\]
is not identically zero. If we denote by $M_E$ the kernel of the evaluation map
\begin{equation}\label{eq:eval-E}
H^0(X,E)\otimes \mathcal O_X\to E,
\end{equation}
then we have an isomorphism
\[
W_q(X,E)^\vee \cong\mathrm{Ker}\Bigl\{H^1(X,\Sym^{q+2}M_E)\to \Sym^{q+2}H^0(X,E)\otimes H^1(X,\mathcal O_X)\Bigr\}.
\]
In particular, if $H^1(X,\mathcal O_X)=0$, then $W_q(X,E)^\vee \cong H^1\bigl(X,\Sym^{q+2}M_E\bigr)$.
\end{thm}

\proof
Based on (\ref{eqn:W}), we know that $W_q(X,E)$ is the middle cohomology of the complex
\[
K\otimes \Sym^qV\stackrel{\delta_{2,q}}{\longrightarrow} V\oo\Sym^{q+1}V \overset{\delta_{1,q+1}}{\lra} \Sym^{q+2}V
\]
where $V=H^0(X,E)^{\vee}$ and $K = \ker(d)^{\perp}$. Dualizing this complex and replacing $K^{\vee}=\im(d)$ by the ambient space $H^0(X,\bw^2 E)$ (which does not affect the middle cohomology), we realize $W_q(X,E)^{\vee}$ as the middle cohomology of a complex
\[
 \Sym^{q+2}H^0(X,E) \lra H^0(X,E)\oo\Sym^{q+1}H^0(X,E) \lra H^0\bigl(X,\bw^2 E\bigr)\oo\Sym^{q}H^0(X,E),
\]
which arises from an alternative construction as follows. Since $M_E$ is resolved by the $2$-term complex (\ref{eq:eval-E}), $\Sym^{q+2}M_E$ is resolved by the $(q+2)$-nd symmetric power of (\ref{eq:eval-E})
\begin{equation}
\label{eq:symPowers}
\Sym^{q+2}H^0(X,E)\otimes \mathcal O_X\to \Sym^{q+1}H^0(X,E)\otimes E\to \Sym^{q} H^0(X,E)\otimes \bigwedge^2E\to \cdots
\end{equation}
and the previous description of $W_q(X,E)^{\vee}$ shows that it coincides with the first cohomology group of the complex obtained from (\ref{eq:symPowers}) by taking global sections.

\vskip 3pt

Since (\ref{eq:symPowers}) resolves $\Sym^{q+2}M_E$, its hypercohomology coincides with the sheaf cohomology of $\Sym^{q+2}M_E$, so we get a spectral sequence
\[E_2^{i,j} = H^i\Bigl(\Sym^{q+2-\bullet}H^0(X,E) \oo H^j(X,\bw^{\bullet}E)\Bigr) \Longrightarrow H^{i+j}(X,\Sym^{q+2}M_E).\]
Since $E_2^{i,j} = 0$ for $i<0$ or $j<0$, it follows that we have an exact sequence
\[
\xymatrix@C=1.5em{0 \ar[r] & E_2^{1,0} \ar[r] & H^1(X,\Sym^{q+2}M_E) \ar[r] \ar@/_2pc/[rr]_{H^1(X,\iota)} & E_2^{0,1} \Big(\ar@{^{(}->}[r] & \Sym^{q+2}H^0(X,E) \otimes H^1(X,\mathcal O_X)\Big)}
\]
where $\iota$ denotes the natural inclusion of $\Sym^{q+2}M_E$ into $ \Sym^{q+2}H^0(X,E)\oo\mc{O}_X$. Since $E_2^{\bullet,0}$ is the complex obtained from (\ref{eq:symPowers}) by taking global sections, we conclude that $E_2^{1,0}=W_q(X,E)^{\vee}$ is the kernel of $H^1(X,\iota)$, and that it is moreover isomorphic to $H^1(X,\Sym^{q+2}M_E)$ when $H^1(X,\mathcal O_X)=0$, as desired.
\endproof

\subsection{Koszul modules associated to $K3$ surfaces.}\label{section:K3}

An important application of Theorem \ref{thm:det} is provided by Lazarfeld-Mukai bundles on $K3$ surfaces. Let $(X,L)$ be a polarized $K3$ surface of genus $g\geq 2$, where $L$ is an ample line bundle of degree $L^2=2g-2$.  We set $H^{\bullet}(X):=H^0(X,\mathbb Z)\oplus H^2(X, \mathbb Z)\oplus H^4(X, \mathbb Z)$. Following \cite{M1}, we define the \emph{Mukai pairing} on $H^{\bullet}(X)$ by
$$(v_0, v_1, v_2)\cdot (w_0,  w_1, w_2):=v_1\cdot w_1-v_2\cdot w_0-v_0\cdot w_2\in H^4(X, \mathbb Z)\cong \mathbb Z.$$
For a sheaf $E$ on $X$, its \emph{Mukai vector} is defined following  \cite[Definition 2.1]{M1}, by setting
\[v(E):=\Bigl(\rk(E),\det(E),\chi(E)-\rk(E)\Bigr)\in H^{\bullet}(X).\]
Note that we have $-\chi(F,F)=v(F)^2$.  We denote by $M_L(v)$ the moduli space of $S$-equivalence classes of $L$-semistable sheaves $E$ on $X$
and having prescribed Mukai vector $v(E)=v$. Let $M_L^s(v)$ the open subset of $M_L(v)$ corresponding to $L$-stable sheaves. It is known that $M_L^s(v)$ is pure dimensional and $\mbox{dim } M_L^s(v)=v^2+2$. Furthermore, if $v^2=-2$, then $M_L(v)=M_L^s(v)$ consists of a single point.

\vskip 3pt

\begin{defn}
A globally generated vector bundle $E$ on a polarized $K3$ surface $(X,L)$ is said to be a \emph{Lazarsfeld-Mukai} bundle if $\mbox{det}(E)\cong L$ and
$H^1(X,E)=H^2(X,E)=0$.
\end{defn}

The Lazarsfeld-Mukai  bundles were  introduced in \cite{L86}, \cite{Mu}, \cite{M1}.  They  can be constructed by choosing a smooth curve $C\in |L|$ and a linear system $A\in W^{r-1}_{d}(C)$ such that both $A$ and $\omega_C\otimes A^{\vee}$ are globally generated, where $r\geq 2$. The dual Lazarsfeld-Mukai bundle sits in the following exact sequence on $X$
$$0\longrightarrow E^{\vee} \longrightarrow H^0(C,A)\otimes \OO_X\stackrel{\mathrm{ev}}\longrightarrow \iota_*A\longrightarrow 0,$$
where $\iota \colon C\hookrightarrow X$ is the inclusion. Dualizing, we obtain the short exact sequence
\begin{equation}\label{lm_sorozat}
0\longrightarrow H^0(C,A)^{\vee}\otimes \OO_X\longrightarrow E\longrightarrow \omega_C\otimes A^{\vee}\longrightarrow 0.
\end{equation}

Then $E$ is a globally generated $L$-stable bundle with $\mbox{det}(E)\cong L$ and
$$h^0(X,E)=h^0(C,\omega_C\otimes A^{\vee})+h^0(C,A)=g-d+r-1,$$
thus $v(E)=(r,L,g-d+r-1)$. We refer to \cite{L86} for all these properties.

\vskip 3pt

To $(X,L)$ and $E$ as above, we consider the Koszul module of the associated Lazarsfeld-Mukai bundle
$$W(X, E):=W\bigl(H^0(X,E)^{\vee}, K\bigr),$$
where $K^{\perp}$ is the kernel of the determinant map $d\colon \bigwedge^2 H^0(X,E)\rightarrow H^0\bigl(X,\bigwedge^2 E\bigr)$.

\begin{lem}
If $\mathrm{Pic}(X)=\mathbb Z\cdot L$, the Koszul module $W(X,E)$ has vanishing resonance.
\end{lem}
\begin{proof}
Two non-proportional sections $s_1$ and $s_2$ of $E$ such that $d(s_1\wedge s_2)=0$ correspond to a locally free subsheaf of rank one $A'$ of $E$ with $h^0(X,A')\geq 2$. Since $\mbox{Pic}(X)=\ZZ \cdot L$, it follows in particular $H^0(X,E\otimes L^{\vee})=0$. Tensoring the sequence (\ref{lm_sorozat}) with $L^{\vee}$ and taking cohomology we obtain a contradiction.
\end{proof}

Let $E$ be a Lazarsfeld-Mukai bundle with Mukai vector $v(E)=(r,L,s)$. Since  $E$ is globally generated, we consider the kernel bundle $M_{E}$ sitting in the
exact sequence
$$0\longrightarrow M_{E}\longrightarrow H^0(X,E)\otimes \mathcal{O}_X\longrightarrow E\longrightarrow 0.$$
Then $M_E^{\vee}$ has Mukai vector $v(M_E^{\vee})=(s,L,r)$. Then $H^1(X,M_E^{\vee})=H^2(X,M_E^{\vee})=0$, furthermore $M_E^{\vee}$ is globally generated and $H^0(X,M_E^{\vee})\cong H^0(X,E)^{\vee}$. In particular, $M_E^{\vee}$ is also a Lazarsfeld-Mukai bundle.

\vskip 4pt

\noindent {\emph{Proof of Theorem \ref{lm_eredmeny}.}} We start with a Lazarsfeld-Mukai bundle $E$ with Mukai vector $v(E)=(r,L,s)$. Then $M_E^{\vee}$ is also a
Lazarsfeld-Mukai bundle with $v(M_E^{\vee})=(s,L,r)$ which has vanishing resonance. Since $h^0(X,M_E^{\vee})=h^0(X,E)=r+s$, the conclusion follows by applying
Theorem \ref{rk2}.

\hfill $\Box$

If $v(E)=(r,L,s)$, a rather lengthy but elementary calculation with Chern classes shows that the symmetric powers of $E$ have Mukai vector
$$
v\bigl(\Sym^b E\bigr)=\Bigl({r+b-1\choose b}, {r+b-1\choose r}L,
$$
$$
{r+b-1\choose b}\frac{b^2(g-r+s-1)-b(r^2+g-sr-1)+r(r+1)}{r(r+1)} \Bigr)\in H^{\bullet}(X).
$$
When $E$ is a spherical object, that is $v^2(E)=-2$, in which case the moduli space $M_L(v)$ consists only of $E$, then $g=rs$ and the above formula becomes more manageable:
\begin{equation}\label{chern_mukai}
v\bigl(\Sym^b E\bigr)=\left({r+b-1\choose b}, {r+b-1\choose r}L,
{r+b-1\choose b}\frac{b^2s-(b-1)(b+r)}{r}\right).
\end{equation}
In particular, Theorem \ref{lm_eredmeny} shows that a general vector bundle $F\in M_L(v)$, where $v$ is the Mukai vector given by (\ref{chern_mukai}),  satisfies $H^1(X,F)=0$. Theorem \ref{lm_eredmeny} is optimal when Theorem \ref{thm:mainresult} is applied
in the divisorial case. We record this result:

\begin{thm}\label{rk2k3}
Let $(X,L)$ be a $K3$ surface of genus $g=2r\geq 4$ with  $\mathrm{Pic}(X)=\ZZ\cdot  L$.
If $E$ is the unique Lazarsfeld-Mukai bundle with vector $v(E)=(r,L,2)$, then $H^1\bigl(X, \Sym^b E\bigr)=0$ for $b\geq r+1$ and
$$h^1\bigl(X,\Sym^b E\bigr)={r+b-1\choose r+1} \frac{r(r-b+1)}{b} \  \  \mbox{ for } b\leq r.$$
\end{thm}
\begin{proof} Apply directly Theorem \ref{lm_eredmeny} coupled with the estimate provided by Theorem \ref{thm:mainresult}.
\end{proof}

\begin{rmk}
Inside the moduli space $\mathcal{F}_g$ of polarized $K3$ surfaces of genus $g$, the locus $\mathfrak{NL}_1$ of those polarized $K3$ surfaces $[X,L]$ for which $H^1(X,\Sym^{r+1} E)\neq 0$ for a vector bundle $E\in M_L^s(r, L,2)$ is a divisor
of Noether-Lefschetz type. Similarly, for $b\geq 1$, the locus $\mathfrak{NL}_b$   of those $[X,L]\in \mathcal{F}_{2r}$ for which $h^1\bigl(X, \Sym^b E)>{r+b-1\choose r+1}\frac{r(r-b+1)}{b}$ is via Theorem \ref{rk2k3} of Noether-Lefschetz type and its class can be computed in terms of the Hodge classes on $\F_g$. Understanding the relative position of the classes $\mathfrak{NL}_b$, in particularly deciding when these loci are empty  will thus lead to non-trivial relations among tautological classes in $CH^{\bullet}(\F_g)$ in the spirit of \cite{FR} or \cite{PY}.
\end{rmk}

Keeping the set-up as above, we fix a general curve $C\in |L|$, therefore $C$ is smooth of genus $2r$ and $W^1_{r+1}(C)$ consists of $\frac{(2r)!}{r!\cdot(r+1)!}$ reduced points, see \cite{L86}. The restriction $E_C$ of the Lazarsfeld-Mukai bundle $E\in M_L(r, L,2)$ is a stable rank $2$ vector bundle with $\mbox{det}(E_C)\cong \omega_C$ and $h^0(C,E_C)=h^0(X,E)=r+2$. Since $h^0(C,\omega_C)<2h^0(C,E_C)-3$, the vector bundle $E_C$ has  non--trivial resonance which we describe below.

\medskip

Put $V=H^0(C,E_C)^\vee$ and $K:=H^0(C,\omega_C)^\vee$ viewed as a subset of $\bigwedge^2V$ via the dual of the map
$d\colon \bigwedge^2 H^0(C,E_C)\rightarrow H^0(C, \omega_C)$. Each pure tensor $[a\wedge b]\in \mathrm{Gr}_2(V^\vee)\cap \mathbf{P}K^\perp$ corresponds to a globally generated subpencil of $E_C$. Without loss of generality, we may assume that the quotient is locally free. We can prove even more:

\begin{lemma}
\label{lem:subpencils}
If  $A$ is a line subbundle  of $E_C$ with $h^0(C, A)\ge 2$, then $A\in W^1_{r+1}(C)$.
\end{lemma}

\begin{proof} The bundle $E_C$ lies in an extension
\[
0\longrightarrow A\longrightarrow  E_C\longrightarrow  \omega_C\otimes A^{\vee}\longrightarrow 0.
\]

Since $E_C$ is globally generated, $\omega_C\otimes A^{\vee}$ is also globally generated and hence either $h^0(C,\omega_C\otimes A^{\vee})\ge 2$ or $A\cong \omega_C$. Since $E_C$ is stable  and $\mu(E_C)=2r-1$, the latter case is ruled out. In particular, $A$ contributes to the Clifford index. On the other hand,
$h^0(C,A)+h^0(C,\omega_C\otimes A^{\vee})\ge h^0(C,E_C)=r+2$ which implies that $\mathrm{Cliff}(A)\le r-1$. Hence either $A$ or its residual $\omega_C\otimes A^{\vee}$ belong to $W^1_{r+1}(C)$. However, the latter case contradicts the stability of $E_C$, hence it does not appear.
\end{proof}

Lemma \ref{lem:subpencils} shows that $\mathrm{Gr}_2(V^\vee)\cap \mathbf{P}K^\perp\cong W^1_{r+1}(C)$ and is finite and moreover $$\bigl|\mathrm{Gr}_2(V^\vee)\cap \mathbf{P}K^\perp\bigr|=\frac{(2r)!}{r!\cdot (r+1)!}.$$

Before stating the next result we recall the various properties of the resonance variety of a Koszul module given in Definition \ref{def:resprop}.
In the case of a vector bundle over a curve, isotropy and separability are related to specific geometric properties. The following result will be used later:

\begin{lem}\label{lem:isotropy}
Let $F$ be a vector bundle of rank $2$ over a smooth curve $C$ and let $\overline{V}^\vee\subseteq H^0(C,F)$.
\begin{itemize}
\item[(i)] $\overline{V}^\vee$ is isotropic if and only if it generates a rank-one subsheaf $B$ inside $F$.
\item[(ii)] Suppose that $E$ is given by an extension of line bundles
\[
0\longrightarrow B\longrightarrow F\longrightarrow B'\longrightarrow 0,
\]
with $B$ globally generated, and denote $\overline{W}:=\mathrm{Im}\bigl\{H^0(C,F)\to H^0(C,B')\bigr\}$. If the multiplication map
\[
\mu\colon H^0(C,B)\otimes \overline{W}\to H^0\bigl(C,\mathrm{det}(F)\bigr)
\]
is injective, then $\overline{V}^\vee=H^0(C,B)$ is strongly isotropic.
\end{itemize}
\end{lem}

\proof
(i) If $L\subseteq F$ is a rank-one subsheaf, then $H^0(C,B)\subseteq H^0(C,F)$ is isotropic, as the restriction of the determinant map to $\bigwedge^2H^0(C,B)$ vanishes identically. Conversely, let $\overline{V}^\vee$ be an isotropic subspace of $H^0(C,F)$.
Then any vector $0\neq a\wedge b\in \bigwedge^2 \overline{V}^\vee$ generates a rank-one subsheaf of $F$. In particular, for a generic point $x\in C$ the vectors $a(x), b(x)\in F(x)$ are linearly dependent, hence the span of $\bigl\{a(x): a\in \overline{V}^\vee\bigr\}$ is one--dimensional.


(ii) Observe first that if $a\in\overline{V}^\vee$, $b\in H^0(C,F)$ and $b'\in \overline{W}$ is the image of $b$, then $\mu(a\otimes b')=d(a\wedge b)$.
Assume $\sum a_i\wedge b_i\in K^\perp\cap \langle\overline{V}^\vee\rangle_{\wedge V^{\vee}}$, with $\{a_i\}$ linearly independent in $\overline{V}^\vee$ and $b_i\in H^0(C,F)$. If $b'_i$ is the image of $b_i$ in $\overline{W}$, we have $\mu(\sum a_i\otimes b_i')=0$. From the hypothesis,  $\sum a_i\otimes b_i'=0$ and, since $a_i$ are independent, $b_i'=0$, that is,  $b_i\in H^0(C,B)$ for all $i$.  In particular, $\sum a_i\wedge b_i\in \bigwedge^2\overline{V}^\vee$.
\endproof

We now return to the set-up when $C\in |L|$ is a curve of genus $2r$ on a $K3$ surface $X$.

\vskip 4pt

\begin{proposition}
\label{prop:genus8}
Fix $C\in |L|$ generic as above. The resonance $\mathcal{R}(C, E_C)$ is strongly isotropic, and its projectivisation $\mathbf{R}(C, E_C)$ is the union of $\frac{(2r)!}{r!\cdot (r+1)!}$ disjoint lines.
\end{proposition}

\proof
From Proposition \ref{prop:isolated} we infer that  that $\mathbf{R}(C,E_C)$ is a union of $\frac{(2r)!}{r!\cdot(r+1)!}$ disjoint lines, all isotropic. In order to establish the separability of these components, $\ell_{ab}$ be a component, corresponding to a point $[a\wedge b]\in \mathrm{Gr}_2(V^\vee)\cap \mathbf{P}K^\perp$ and denote by $\overline{V}^\vee$ the subspace in $V^\vee$ generated by $a$ and $b$. If $A$ denotes the subpencil of $E_C$ generated by $a$ and $b$,
then $E_C$ is presented as an extension
\[
0\longrightarrow  A\longrightarrow  E_C\longrightarrow \omega_C\otimes A^{\vee}\longrightarrow  0
\]
Since the Petri map $\mu\colon H^0(C, A)\otimes H^0(C, \omega_C\otimes A^{\vee})\to H^0(C, \omega_C)$ is injective, we can apply Lemma \ref{lem:isotropy} (ii) to conclude.
\endproof

\subsection{Koszul modules associated to $K3$ surfaces of odd genus.} Using a variation compared to the even genus case, one can also associate to a general $K3$ surface of \emph{odd} genus a
Koszul module $W(V,K)$ in the divisorial case $\mbox{dim}(K)=2\ \mbox{dim}(V)-3$ as follows.

\vskip 3pt

Fix a polarized $K3$ surface $[X,L]$ of odd genus $g=2r+1$ such that $\mbox{Pic}(X)=\ZZ \cdot L$ and choose a smooth curve
$C\in |L|$. Recall that  $\widehat{X}:=M_L(2,X,r)$ is  the Fourier-Mukai partner of $X$. Denoting by
$\mathcal{SU}_C(2,\omega, r+2)$ the moduli space of $S$-equivalence classes of semistable rank $2$ vector bundles $E_C$ on $C$ with $\mbox{det}(E_C)\cong \omega_C$ and $h^0(C,E_C)\geq r+2$, the restriction map induces an isomorphism, see \cite{ABS1},
$$\widehat{X} \cong \mathcal{SU}_C(2,\omega,r+2), \ \ \mbox{ } \ E\mapsto E_C.$$
Moreover, it can be shown that $X$ as the Fourier-Mukai dual of $\mathcal{SU}_C(2, \omega, r+2)$ is the only $K3$ surface containing $C$ as long as $s$ is odd, see \cite{ABS1}, \cite{Fe}.

The Brill-Noether locus $W^1_{r+2}(C)$ is a smooth curve (recall that $C$ satisfies Petri's Theorem \cite{L86}) and we have the following formula for its genus \cite{EH87}:
\begin{equation}
g\bigl(W^1_{r+2}(C)\bigr)=1+\frac{r}{r+1}{2r+2\choose r}.
\end{equation}
Using \cite{V1}, one has  a map
$j\colon W^1_{r+2}(C)\rightarrow \widehat{X}$ which associates to $A\in W^1_{r+2}(C)$ the  rank $2$ Lazarsfeld-Mukai vector bundle $E_{C,A}$
defined by the sequence (\ref{lm_sorozat}). Its restriction $E_A:=E_{C,A}\otimes \OO_C$ to $C$ satisfies  $h^0(S,E_{C,A})=h^0(C, E_A)=r+2$.

\vskip 4pt

To a pair $(C,E_C)$, where $C\in |L|$ and $E_C\in \mathcal{SU}_C(2,\omega,r+2)$, using Definition \ref{ex:W(X,E)} we associate the Koszul module
$W(C,E_C):=W(V,K)$ and its resonance variety $\mathcal{R}(C, E_C)$. Note that since $h^0(C,\omega_C)=2h^0(C,E)-3$, we are in the divisorial case of Theorem \ref{thm:mainresult}. We denote by $M_{E_C}$ the kernel of the evaluation map $H^0(C,E_C)\otimes \OO_C\rightarrow E_C$.

\begin{thm}
One has a canonical identification $$j\bigl(W^1_{r+2}(C)\bigr)\cong \bigl\{E_C\in \mathcal{SU}_C(2,\omega, r+2):\mathcal{R}(C,E_C)\neq 0\bigr\}.$$
Furthermore, for each $E_C\in \mathcal{SU}_C(2,\omega, r+2)\setminus j\bigl(W^1_{r+2}(C)\bigr)$ the map
\begin{equation}\label{kgorbe}
H^1\bigl(C,\Sym^{r+1} M_{E_C}\bigr)\rightarrow \Sym^{r+1} H^0(C,E_C)\otimes H^1(C,\OO_C)
\end{equation}
is injective.
\end{thm}

The geometric meaning of the injectivity of the map (\ref{kgorbe}) is mysterious and requires further study. In what follows we will prove Theorem
\ref{thm:classVoisin}.

\vskip 3pt

Recall that $[X, L]\in \mathcal{F}_{2r+1}$ with $\mbox{Pic}(X)=\mathbb Z\cdot L$ and we consider the projections
\[
\xymatrix{
  & X\times \widehat{X} \ar[dl]_{\pi_1} \ar[dr]^{\pi_2} & \\
   X & & \widehat{X},       \\
                 }
\]
and denote by $\P$ a Poincar\'e bundle of rank $2$ on $X\times \widehat{X}$. \footnote{A Poincar\'e bundle $\P$ exists only when $g\equiv 3 \  \mathrm{mod} \ 4$,  that is, when $r$ is odd. When $r$ is even, the class $\phi$ is divisible by two and $\P$ does not exist globally. As pointed out in \cite{M2}, one has to take instead the universal $\PP^1$-bundle corresponding to $\PP(\P)$ (which does exist) and carry out the calculation of the class of the curve $\mathcal{R}(\widehat{X}, C)$ at that level. Theorem \ref{thm:classVoisin} remains valid independent of the parity of $r$.}
One writes
$$c_1(\P)=\pi_1^*h+\pi_2^*\varphi\in \pi_1^*H^2(X)\oplus \pi_2^*H^2(\widehat{X}) \ \ \mbox{ and  } \ \ c_2^{\mathrm{mid}}(\P)\in \pi_1^*H^2(X)\otimes \pi_2^*H^2(\widehat{X})$$
for the first Chern class of $\P$ respectively the middle class in the K\"unneth decomposition of $c_2(\P)$. Following \cite{M2} we define the class
$\psi\in H^2(\widehat{X})$ characterized by the property $\pi_1^*h\cdot c_2^{\mathrm{mid}}(\P)=[\mathrm{pt}]\otimes \pi_2^*\psi \in \pi_1^* H^4(X)\oplus \pi_2^*H^2(\widehat{X})$, where
$[\mathrm{pt}]$ is the fundamental class of $X$. It is also shown in \cite{M2} that if one sets
\begin{equation}\label{eq:dualpol}
\widehat{h}:=\psi-2r\cdot \varphi \in H^2(\widehat{X}),
\end{equation}
then $\widehat{h}$ is a polarization on $\widehat{X}$ satisfying $\widehat{h}^2=h^2=2g-2=4r$.
We now introduce the following vector bundles on $\widehat{X}$ having as fibres over a point $[E]$ the spaces $H^0(X, E)$ and
$H^0\bigl(C,\mbox{det}(E_C)\bigr)\cong H^0(C, \omega_C)$ respectively, that is,
$$\E:=(\pi_2)_*\bigl(\P\bigr)\ \ \ \mbox{ and } \ \ \ \F:=(\pi_2)_*\bigl(\bigwedge^2 \P_{|C\times \widehat{X}}\bigr).$$

\begin{prop}\label{prop:Chernclasses}
The following formulas hold in $H^2(\widehat{X})$:

$$c_1(\F)=(2r+1)\varphi   \ \ \ \mbox{  and  } \ \ \   c_1(\E)=\frac{3r+2}{2}\varphi-\frac{\psi}{2}.$$
\end{prop}
\proof
We apply Grothendieck-Riemann-Roch to the map $\pi_2$ and the sheaf $\P$ using that $\bigl(R^i\pi_2)_*(\P)=0$ for $i=1,2$, since $H^1(X,E)=H^2(X, E)=0$, for
$[E]\in \widehat{X}$. We thus write

\begin{multline*}
c_1(\E)=c_1\Bigl((\pi_2)_*(\P)\Bigr)=(\pi_2)_*\Bigl[\Bigr(2+c_1(\P)+\frac{c_1^2(\P)-2c_2(\P)}{2}+\frac{c_1^3(\P)-3c_1(\P)c_2(\P)}{6}\Bigr)\cdot
\\
\Bigl(1+
\frac{\pi_1^*c_2(\Omega_X)}{12}\Bigr)\Bigr]_3=\frac{1}{12}(\pi_2)_*\bigl(c_1(\P)\cdot \pi_1^*c_2(\Omega_X)\bigr)+\frac{1}{6}(\pi_2)_*\bigl(c_1^3(\P)\bigr)-\frac{1}{2}(\pi_2)_*\bigl(c_1(\P)\cdot c_2(\P)\bigr).
\end{multline*}

Observe that $(\pi_2)_*\bigl(c_1(\P)\cdot \pi_1^*c_2(\Omega_X)\bigr)=(\pi_2)_*\bigl((\pi_1^*(h)+\pi_2^*(\varphi))\cdot \pi_1^*c_2(\Omega_X)\bigr)=24\varphi$. Furthermore, one also has
$$(\pi_2)_*\bigl(c_1^3(\P)\bigr)=(\pi_2)_*\bigl(3\pi_1^*(h^2)\cdot \pi_2^*\varphi\bigr)=6(g-1)=12r\cdot \varphi,$$
whereas using the K\"unneth decomposition $(\pi_2)_*\bigl(c_1(\P)\cdot c_2(\P))=(r+2)\varphi+\psi$. Substituting, we obtain the claimed formula for
$c_1(\E)$. The calculation of $c_1(\F)$ is analogous by applying Grothendieck-Riemann-Roch to the pushforward of $\mathrm{det}(\P)$ under $\pi_2$.
First we compute that $c_1\bigl((\pi_2)_*(\bigwedge^2 \P))=(2r+2)\cdot \varphi$, then from the exact sequence on $\widehat{X}$
$$0\longrightarrow (\pi_2)_*(\OO_{X\times \widehat{X}})\longrightarrow (\pi_2)_*(\P)\longrightarrow \F\longrightarrow 0,$$
since $c_1\bigl(\pi_2)_*\OO_{X\times \widehat{X}}\bigr)=\varphi$, we obtain $c_1(\F)=(2r+1)\varphi$, as claimed.
\endproof

\vskip 4pt

\noindent \emph{Proof of Theorem \ref{thm:classVoisin}.} One has a morphism of vector bundle $\phi\colon \bigwedge^2 \E\rightarrow \F$ over $\widehat{X}$, whose fibre over a point $[E]$ is precisely the determinant map $d\colon \bigwedge^2 H^0(X,E)\rightarrow H^0(C,\omega_C)$. Noting that $\mbox{rk}(\E)=r+2$ and $\mbox{rk}(\F)=2r+1$, via the terminology of Theorem \ref{thm:alt}, the resonance divisor $\mathfrak{Res}(\phi)$ of the morphism $\phi$ can be identified with the Voisin curve
$\mathcal{R}(\widehat{X}, C)$. Using Theorem \ref{thm:alt} we thus find
$$[\mathcal{R}(\widehat{X}, C)]=\frac{(2r)!}{r!\cdot (r+1)!}\Bigl(c_1(\F)-\frac{4r+2}{r+2}c_1(\E)\Bigr)=\frac{(2r)!}{r!\cdot (r+1)!} \Bigl(\frac{2r+1}{r+2}\psi-\frac{2r(2r+1)}{r+2}\varphi\Bigr),$$
which yields precisely the predicted formula.
\hfill $\Box$

\begin{rmk}
It is natural to conjecture that for a general $C\in |L|$, the singularities of the curve $\mathcal{R}_{\widehat{X}, C}$ are nodal. Proving this seems challenging even for small $r$.
\end{rmk}

\section{Gaussian Koszul modules and thickenings of algebraic varieties}

An important class of Koszul modules where the triviality of resonance is automatically  satisfied is given by the \emph{Gaussian maps} \cite{Wahl} on projective varieties. Let $L$ be a line bundle on a smooth complex projective variety $X$. The Gaussian Wahl map
$$\psi_L:\bigwedge^2 H^0(X,L)\rightarrow H^0(X,\Omega_X^1\otimes L^{2}),$$
is locally defined by $\psi_L(\sum_{i} f_i\wedge g_i):=\sum_i (f_i\cdot dg_i-g_i\cdot df_i)$, for  $f_i, g_i\in H^0(X,L)$.

\vskip 3pt

If $X$ is a smooth curve and $L=\mathcal{O}_X(1)$, the map $\psi_L$ is given by associating to a point  $p\in X$ the projectivized tangent line $\mathbf T_p(X)\in \Gr_2\bigl(H^0(X,L)^{\vee}\bigr)$ under the Pl\"ucker embedding of the Grassmannian of lines.

\begin{defn}\label{gaussian_module}
Set $V:=H^0(X,L)^{\vee}$ and $K^{\perp}:=\mbox{Ker}(\psi_L)$. The associated Koszul module
$$\mathcal{G}(X,L):=W(V,K),$$
is called the \emph{Gaussian module} of the pair $(X,L)$.
\end{defn}

Since $\psi_L(f\wedge g)=0$ if and only if $d\bigl(\frac{f}{g}\bigr)=0$, it follows that $\psi_L$ is injective on decomposable tensors, therefore $\mathcal{R}(V,K)=\{0\}$. In particular, $\mbox{rk}(\psi_L)\geq 2h^0(X,L)-3$. If $X$ is a smooth curve, the equality $\mbox{rk}(\psi_L)=2h^0(X,L)-3$ holds if and only if the image of $X$ under the linear system $|L|$ is a rational normal curve, see \cite[Theorem 1.3]{CilM}.

\vskip 4pt

We introduce the vector bundle $R_L$ defined by the  exact sequence
\begin{equation}\label{eq:xi-eta}
0 \lra R_L \overset{\iota}{\lra} H^0(X,L)\oo \mc{O}_X \overset{r}{\lra} J_1(L)\lra 0,
\end{equation}

where $J_1(L)$ is the first jet bundle of $L$. The map $r$ in (\ref{eq:xi-eta}) can be defined locally by $r(w) = (dw,w)$. We consider the exact sequence
\begin{equation}\label{eq:def-eta}
 0 \lra \Omega^1_X \oo L \lra J_1(L) \overset{p}{\lra} L \lra 0
\end{equation}
and observe that $p\circ r$ is the evaluation map $H^0(X,L) \oo \mc{O}_X \rightarrow L$. In particular, one also has the following exact sequence on $X$:
\begin{equation}\label{RLML}
0\lra R_L\lra M_L\lra \Omega_X^1 \otimes L\lra 0.
\end{equation}

In case $L$ is very ample and we consider the embedding $\varphi_L\colon X\hookrightarrow\bb{P}(V)$ defined by $V=H^0(X,L)$ and write $\I$ for the ideal sheaf of $X$ in this embedding, we have $R_L=N_L^{\vee}\otimes L= \I/\I^2 \oo L$.

\vskip 3pt

From (\ref{eq:def-eta}) we obtain an induced exact sequence
\begin{equation}\label{eq:seq-w2eta}
 0 \lra \Omega^2_X \oo L^{2} \lra \bw^2 J_1(L) \overset{a}{\lra} \Omega^1_X \oo L^{2} \lra 0,
\end{equation}
and consider the composition
\[ a\circ\bw^2 r \colon \ \bw^2 H^0(X,L) \oo \mc{O}_X \lra \Omega^1_X \oo L^{2}.\]
The induced map on global sections is the Gaussian map $\psi_L$.  Our goal is to give a cohomological interpretation of the graded components of the Koszul module $\mathcal{G}(V,K)$. To that end, we fix $q\geq 0$ and consider the composition
\begin{equation}\label{eq:ximap}
s\colon \ \Sym^{q+2} R_L \lra \Sym^{q+1} R_L \oo R_L \lra \Sym^{q+1} H^0(X,L) \oo R_L,
\end{equation}
where the first one is the natural inclusion, and the second map is $\Sym^{q+1}(\iota) \oo \op{id}_{R_L}$.

\begin{thm}\label{thm:koszul-Wahl}
 For each $q\geq 0$, the components of the Gausssian module $\mathcal{G}(X,L)$ are given by
 \[\mathcal{G}_q(X,L)^{\vee} = \ker\Bigl\{H^1\bigl(X,\Sym^{q+2}\ R_L\bigr) \overset{t}{\lra} \Sym^{q+1} H^0(X, L) \oo H^1(X, R_L)\Bigr\},\]
 where the map $t=H^1(X,s)$ is induced by (\ref{eq:ximap}).
\end{thm}

\noindent To prove the theorem we first show that $K^{\perp}$ is also equal to $\ker\bigl(H^0(X,\bw^2 r)\bigr)$:

\begin{lemma}\label{lem:koszul-bw2eta}
The restriction of the map $\alpha=H^0(X,a)$ to the image of $\beta=H^0(X,\bw^2 r)$ is injective. In particular we have that $\ker(\psi_L) = \ker(\beta)$.
\end{lemma}

\begin{proof}
 Since all the sheaves involved are locally free, it suffices to localize at the generic point of $X$ and show that $\alpha$ is injective on the image of $\beta$ there. In particular we may choose a local generator of $L$ and identify elements of $W$ with rational functions on $X$. We have that $\beta(w \wedge w') = (dw \wedge dw',w\cdot dw' - w'\cdot dw)$
 and $\alpha$ is the projection onto the second component. If $x=\sum_i w_i\cdot w_i'\in \ker(\alpha\circ\beta)$, then
 \[ 0 = \alpha(\beta(x)) = \sum_i (w_i \cdot dw_i' - w_i' \cdot dw_i).\]
 Differentiating this equality shows that $\sum_i dw_i \wedge dw_i'=0$, so that $\beta(x)=0$, that is, $\alpha$ is injective on $\op{Im}(\beta)$, as desired.
\end{proof}

Using (\ref{eqn:W}) and Lemma~\ref{lem:koszul-bw2eta} we get that $\mathcal{G}_q(X,L)^{\vee}$ is the middle cohomology of
\begin{equation}\label{eq:complex-Wq-Wahl}
\Sym^{q+2}H^0(L) \lra \Sym^{q+1}H^0(L) \oo H^0(L) \lra \Sym^q H^0(L) \oo H^0\bigl(X, \bw^2 J_1(L)\bigr).
\end{equation}
The second map takes the differential $\Sym^{q+1}H^0(L) \oo H^0(L) \to \Sym^q H^0(L) \oo \bw^2 H^0(L)$ and composes it with $\op{id}_{\Sym^q H^0(L)} \oo H^0(X,\bw^2 r)$.

\begin{proof}[Proof of Theorem~\ref{thm:koszul-Wahl}] It follows from (\ref{eq:xi-eta}) that for $q\geq 0$ we have an exact sequence
\begin{equation}\label{eq:resolution-Symxi}
0 \to \Sym^{q+2} R_L \to \Sym^{q+2} H^0(L) \oo \mc{O}_X \to \Sym^{q+1} H^0(L) \oo J_1(L) \twoheadrightarrow \Sym^q H^0(L) \oo \bw^2 J_1(L).
\end{equation}
Dropping the first term and taking global sections we obtain the middle row of the commutative diagram
\[
\xymatrix{
\Sym^{q+2} H^0(L) \ar[r] \ar@{=}[d] & \Sym^{q+1} H^0(L) \oo H^0(L) \ar[r] \ar[d]^{\op{id} \oo\,H^0(X,r)} & \Sym^q H^0(L) \oo H^0\bigl(\bw^2 J_1(L)\bigr) \ar@{=}[d] \\
\Sym^{q+2} H^0(L) \ar[r] \ar[d] & \Sym^{q+1} H^0(L) \oo H^0(J_1(L)) \ar[r] \ar[d]^{\op{id} \oo\,\delta} & \Sym^q H^0(L) \oo H^0\bigl(\bw^2 J_1(L)\bigr) \ar[d] \\
0 \ar[r] & \Sym^{q+1} H^0(X, L) \oo H^1(X, R_L) \ar[r] & 0 \\
}
\]
where the first row of the diagram is given by (\ref{eq:complex-Wq-Wahl}), and $\delta$ is the connecting homomorphism associated with the long exact sequence in cohomology of (\ref{eq:xi-eta}). Since $p\circ r$ is the evaluation map $H^0(X,L) \oo \mc{O}_X \rightarrow L$, we get that  $H^0(X,r)$ is injective.

\vskip 3pt

If we write $H$ for the middle homology of the middle row of the above diagram, it follows from (\ref{eq:complex-Wq-Wahl}) that
\[ \mathcal{G}_q(X,L)^{\vee} = \ker \Bigl\{H \overset{u}{\lra} \Sym^{q+1}H^0(X,L) \oo H^1(X,R_L)\Bigr\},\]
where the map $u$ is induced by $\op{id} \oo\,\delta$. Just as in Theorem~2.8, we shall realize $H$ as
\begin{equation}\label{eq:C-is-ker}
 H = \ker\Bigl\{H^1\bigl(X,\Sym^{q+2}R_L\bigr) \overset{v}{\lra} \Sym^{q+2} H^0(X,L) \oo H^1(X,\mc{O}_X)\Bigr\},
\end{equation}
so $H$ can be thought of as a subgroup of $H^1\bigl(X,\Sym^{q+2}R_L\bigr)$. Under this identification, we claim that $u$ is the restriction of $t$ to $H$. Moreover, $v$ factors through $t$, therefore $H = \ker(v) \supseteq \ker(t)$ and $\ker(u)=\ker(t)$ as desired.

\vskip 3pt

In order to see that $v$ factors through $t$, we consider the diagram
\[
\xymatrix{
\Sym^{q+2}R_L \ar[rr]^{\Sym^{q+2}(\iota)} \ar[dr]_s & & \Sym^{q+2}H^0(X,L) \oo \mc{O}_X \\
& \Sym^{q+1}H^0(X,L) \oo R_L \ar[ur]_o & \\
}
\]
(commutative up to multiplication by a non-zero scalar) where the map $o$ is induced by~$\iota$ and the multiplication $\Sym^{q+1}H^0(X,L) \oo H^0(X,L) \lra \Sym^{q+2} H^0(X,L)$. Since $v=H^1(X,\Sym^{q+2}(\iota))$ and $t=H^1(X,s)$, it follows that $v$ factors through $t$.

\vskip 3pt

To prove the assertion (\ref{eq:C-is-ker}) and that $u=t_{|H}$, we split (\ref{eq:resolution-Symxi}) into short exact sequences
\[0 \lra \Sym^{q+2}R_L \lra \Sym^{q+2}H^0(X,L) \oo \mc{O}_X \lra M \lra 0,\mbox{ and }\]
\[0\lra M \overset{j}{\lra} \Sym^{q+1}H^0(X,L) \oo J_1(L) \lra \Sym^q H^0(X,L) \oo \bw^2 J_1(L) \lra 0.\]
By construction, $H$ is the cokernel of the map $\Sym^{q+2}H^0(X,L)\rightarrow H^0(X,M)$, which is the same as the kernel of
\[H^1(X,\Sym^{q+2}R_L) \lra  \Sym^{q+2}H^0(X,L) \oo H^1(X,\mc{O}_X).\]

We consider the commutative diagram (where $\Delta$ is the natural inclusion)
\[
\xymatrix{
  \Sym^{q+2} R_L \ar[r] \ar[d]_s & \Sym^{q+2}H^0(L) \oo \mc{O}_X \ar[r] \ar[d]_{\Delta} & M \ar[d]_j  \\
 \Sym^{q+1} H^0(L) \oo R_L \ar[r] & \Sym^{q+1}H^0(L) \oo H^0(L) \oo \mc{O}_X \ar[r] & \Sym^{q+1}H^0(L) \oo J_1(L)   \\
}
\]
which gives rise by taking cohomology to a commutative diagram
\[
\xymatrix{
 H^0(X,M) \ar[d] \ar[r] & H^1(X, \Sym^{q+2}R_L) \ar[d]^t \\
 \Sym^{q+1}H^0(X,L) \oo H^0(X,J_1(L)) \ar[r]^{\op{id} \oo\,\delta} & \Sym^{q+1}H^0(X,L) \oo H^1(X,R_L) \\
}
\]
Since $u$ was induced by $\op{id} \oo\,\delta$, it follows that after identifying $H$ with a subgroup of $H^1(X, \Sym^{q+2} R_L)$ we get that $u$ is the restriction of $t$, concluding the proof.
\end{proof}

Theorem \ref{thm:koszul-Wahl} has a more transparent geometric interpretation under suitable assumptions.

\begin{cor}\label{cor:wahl}
Let $L$ be a very ample line bundle on a smoooth projective variety $X$ such that $q(X)=0$ and $H^0(X,\Omega_X^1\otimes L)=0$. Then
$H^1\bigl(X,\mathrm{Sym}^b R_L\bigr)=0$ for all $b\geq r(L)$.
\end{cor}
\begin{proof}
The hypothesis $q(X)=0$ implies $H^1(X,M_L)=0$. From the exact sequence (\ref{RLML}), we obtain that
$H^0(X,\Omega_X^1\otimes L)\cong H^1(X,R_L)$, therefore $H^1(X,R_L)=0$ as well. The conclusion now follows by applying
Theorem \ref{thm:koszul-Wahl} coupled with Theorem \ref{thm:mainresult}.
\end{proof}

One can reformulate these results in terms of stabilization of cohomology on the successive thickenings of the subvariety $X\subseteq \PP^r$. For $b\geq 0$, we denote by $X_b\subseteq \PP^r$ the subscheme defined by the ideal $\I^{b+1}\subseteq \OO_{\PP^r}$, thus we have a system of subschemes
$$X=X_0\rightarrow X_1\rightarrow \cdots \rightarrow X_{b-1}\rightarrow X_{b}\rightarrow \cdots.$$

\noindent  \emph{Proof of Theorem \ref{infneigh}}. Suppose the projective variety $X\subseteq \PP^r$ is embedded by the line bundle $L:=\OO_X(1)$. Then
for each $b\geq 1$, one has the short exact sequence
\begin{equation}\label{infseq}
0\lra \mathrm{Sym}^b N_L^{\vee}\lra \OO_{X_b}\lra \OO_{X_{b-1}}\lra 0.
\end{equation}
Since $R_L=N_L^{\vee}\otimes L$, tensoring this exact sequence with $L^b$ and taking cohomology, we obtain from Corollary \ref{cor:wahl}
that the map $H^0\bigl(X_b,\OO_{X_b}(b)\bigr)\rightarrow H^0\bigl(X_{b-1}, \OO_{X_{b-1}}(b)\bigr)$ is surjective for $b\geq r=r(L)$.
The map is also injective, for $H^0(X,\Sym^b R_L)=0$, because of the injectivity of the map $\Sym^{b} H^0(X,L)\rightarrow \Sym^{b-1} H^0(X,L)
\otimes H^0(X, J_1(L))$, where we use that our assumptions force the map $H^0(r): H^0(X,J_1(L))\rightarrow H^0(X,L)$ induced by the sequence (\ref{eq:xi-eta}) to be an isomorphism.

\vskip 3pt

We assume now that $0\leq a<b$ and set $c:=b-a\geq 1$. To complete the proof we have to show that $H^i(X,\Sym^b R_L\otimes L^{-c})=0$, for $i=0,1$.
To that end, we use the notation from the proof of Theorem \ref{thm:koszul-Wahl}. By Kodaira vanishing $H^1(X,L^{-c})=0$, thus it follows that
there is a surjection $H^0(X, M\otimes L^{-c})\twoheadrightarrow H^1(X,\Sym^b R_L\otimes L^{-c})$. Furthermore, we have an injection
$H^0(X,M\otimes L^{-c})\hookrightarrow \Sym^{b-1} H^0(X,L)\otimes H^0\bigl(X, J_1(L)\otimes L^{-c}\bigr)$.

\vskip 3pt

In order to show that this last cohomology group vanishes, we use the sequence (\ref{eq:def-eta}). Since $H^0(X,\Omega_X^1)=0$, clearly $H^0(X,J_1(L)\otimes L^{-c})=0$, for $c\geq 2$.
For $c=1$, the existence of a non-zero section in $H^0(X,J_1(L)\otimes L^{\vee})$ implies that the sequence (\ref{eq:def-eta}) is split. But this is impossible, for it is known that the \emph{Atiyah class} $\eta_L\in \mbox{Ext}^1(L,\Omega_X^1 \otimes L)\cong H^1(X,\Omega_X^1)$ expressing
$J_1(L)$ as an extension in the sequence  (\ref{eq:def-eta}) equals precisely the Chern class $c_1(L)$. Since $L$ is very ample, this class cannot be zero.

\vskip 3pt

Finally, in order to show that $H^0(X, \Sym^b R_L\otimes L^{-c})=0$, observe that one has an injection $H^0(X,\Sym^b R_L\otimes L^{-c})\hookrightarrow \Sym^b H^0(X,L)\otimes H^0(X, L^{-c})$.

\hfill $\Box$

\subsection{Ramification divisors of canonical pencils.}

We now prove Theorem \ref{thm:canpencil}. We use throughout the standard notation \cite{AC} for the tautological and boundary classes on $\mm_{g,n}$. We consider the universal curve $\pi\colon \mm_{g,n+1}\rightarrow \mm_{g,n}$ endowed with its $n$ tautological sections
whose images we identify with the boundary divisors $\Delta_{0:i, n+1}$ on $\mm_{g,n}$ for $i=1, \ldots, n$. We consider the Hodge bundle
$\E:=\pi_*(\omega_{\pi})$ and the rank $n$ vector bundle
$$\F:=\pi_*\bigl(\omega^3_{\pi|\Delta_{0:1, n+1}+\cdots+\Delta_{0:n,n+1}}\bigr).$$

One has a morphism $\phi\colon \bigwedge^2 \E\rightarrow \F$ which fibrewise is given by the composition
$$0\lra \bigwedge^2 H^0(C, \omega_C)\stackrel{\psi_{\omega_C}}\lra H^0(C, \omega_C^3)\stackrel{\mathrm{res}}\lra H^0\bigl(C, \omega^3_{C|x_1+\cdots+x_n}\bigr).$$

Observe that $\phi_{[C, x_1, \ldots, x_n]}(s_1\wedge s_2)=0$ for $0\neq s_1\wedge s_2\in \bigwedge^2 H^0(C, \omega_C)$ if and only if $x_1+\cdots+
x_n$ lies in the ramification divisor of the cover $C\rightarrow \PP^1$ induced by $s_1$ and $s_2$.

For our next result, recall that $\psi_i$ denotes the class of the line bundle on $\mm_{g,n}$ having as fibre over a point
$[C, x_1, \ldots, x_n]$ the cotangent space $T_{x_i}^{\vee}(C)$, for $i=1, \ldots, n$.

\begin{prop}\label{prop:ramF}
One has $c_1(\F)=3(\psi_1+\cdots+\psi_n)$.
\end{prop}
\proof
One uses the following exact sequence of sheaves on the universal curve $\mm_{g,n+1}$:
\begin{equation}\label{eq:Fseq}
0\lra \pi_*\Bigl(\omega_{\pi}^3(-\sum_{i=1}^{n}\Delta_{0:i,n+1})\Bigr)\lra \pi_*\bigl(\omega_{\pi}^3\bigr)\lra \F\lra
R^1\pi_*\Bigl(\omega_{\pi}^3(-\sum_{i=1}^{n}\Delta_{0:i,n+1})\Bigr)\lra 0.
\end{equation}
On the one hand we use that $c_1\bigl(\pi_*(\omega_{\pi}^3)\bigr)=\lambda+3\kappa$, where $\kappa:=\pi_*\bigl(c_1^2(\omega_{\pi})\bigr)$, see \cite{AC}, on the other hand after an easy application of Grothendieck-Riemann-Roch we can write:
\begin{multline*}
c_1\Bigl(\pi_{!}\bigl(\omega_{\pi}^3(-\sum_{i=1}^n \delta_{0:i, n+1})\bigr)\Bigr)=\pi_*\Bigr[(1+3c_1(\omega_{\pi})-\sum_{i=1}^n \delta_{0:i, n+1}+\frac{\bigl(3c_1(\omega_{\pi})-\sum_{i=1}^n \delta_{0:i,n+1}\bigr)^2}{2}\Bigr)\cdot\\
\Bigl(1-\frac{c_1(\omega_{\pi})}{2}+\frac{c_1^2(\omega_{\pi})+c_2(\Omega_{\pi})}{12}\Bigr)\Bigr]_2=\lambda+3\kappa-3\sum_{i=1}^n\psi_i,
\end{multline*}
where we have used the formulas $\pi_*(\delta_{0:i, n+1}\cdot \delta_{0:j,n+1})=0$ for $i\neq j$, $\pi_*(\delta_{0:i, n+1}^2)=-\psi_i$, as well as the fact that $c_2(\Omega_{\pi})$ can be identified with the codimension $2$ locus of nodes inside $\mm_{g,n+1}$, hence $\pi_*\bigl(c_2(\Omega_{\pi})\bigr)=\delta$, this being the class of the total boundary of $\mm_{g,n}$. This leads to the claimed formula by using the sequence (\ref{eq:Fseq}).
\endproof

\vskip 3pt

\noindent \emph{Proof of Theorem \ref{thm:canpencil}.} We apply Theorem \ref{thm:alt} to the morphism  $\phi\colon \bigwedge^2 \E\rightarrow \F$ on vector bundles on $\mm_{g,2g-3}$, using Proposition \ref{prop:ramF} and that $c_1(\E)=\lambda$.
\hfill $\Box$

\section{Resonance, stability and split bundles}
\label{section:split}

In this section, we prove that important intrinsic properties of bundles of sufficiently large degree on a curve, such as instability or being split, can be read off its resonance. We use the following notation, for a given vector bundle $E$ on a curve $C$ and an integer~$k$

\[
\mathcal{R}_{\ge k}(C,E):=\bigl\{a\in \mathcal{R}(C,E): \exists L\subseteq E \mbox{ line bundle, }\mathrm{deg}(L)\ge k, h^0(L)\ge 2, a\in H^0(L)\bigr\}
\]

By projectivization, these closed loci provide us with a stratification of the projectivized resonance. Indeed, $\mathbf{R}_{\ge k}(C,E)\supseteq \mathbf{R}_{\ge (k+1)}(C,E)$, $\mathbf{R}_{\ge k}(C,E)=\emptyset$ for $k\gg0$, and $\mathbf{R}_{\ge d}(C,E)=\mathbf{R}(C,E)$ if $d$ equals te gonality of the curve. We call this stratification the \emph{degree stratification}.

\begin{thm}
	\label{thm:split}
	Let $E$ be a globally generated rank $2$ vector bundle on a smooth curve $C$ of genus $g\ge 1$.
	Assume that $\mathrm{deg}(E)\ge 4g+1$ and $H^1(C,E)=0$. Then
	
	\begin{itemize}
		\item[(i)]
		$E$ is not stable (respectively unstable) if and only if  $H^0(C,E)^\vee$ has an isotropic subspace of dimension at least $\frac{h^0(E)}{2}$ (respectively $>\frac{h^0(E)}{2}$).
		\item[(ii)]
		$E$ splits as a sum of line bundles $E=N\oplus M$ with $h^0(N\otimes M^\vee)=h^0(M\otimes N^\vee)=0$ if and only if there exist an integer $k$ and isotropic subspaces $V_1^\vee,V_2^\vee\subseteq H^0(C,E)^\vee$ of dimension $\ge 2$ such that $H^0(C,E)^\vee=V_1^\vee\oplus V_2^\vee$ and $\mathcal{R}_{\ge k}(C, E)= V_1^\vee\cup V_2^\vee$.
	\end{itemize}
\end{thm}

\proof
\emph{(i)}
Assume $E$ is not stable, and let $L\subseteq E$ be a maximal destabilizing line subbundle. Since $\mathrm{deg}(L)\ge \frac{\mathrm{deg}(E)}{2}\ge 2g$, the bundle $L$ is non--special and globally generated. Therefore $h^0(C, L)\ge \frac{h^0(C, E)}{2}$ and $H^0(C, L)\subseteq H^0(C, E)$ is isotropic.

\vskip 4pt

Conversely, let us assume $U\subseteq H^0(C, E)$ is isotropic of dimension at least $\frac{h^0(E)}{2}$. Then $U$ generates a line bundle $N\subseteq E$. If $N$ is non-special, then it destabilizes $E$.
If $N$ is special, by Clifford's Theorem
\[
\mathrm{deg}(N)\ge 2h^0(C, N)-2\ge h^0(C, E)-2.
\]
Furthermore, by Riemann-Roch Theorem we obtain
$h^0(E)-2 = \mathrm{deg}(E)-2g\geq \frac{\mathrm{deg}(E)}{2}$. In conclusion, $N$ destabilizes $E$. The unstable case goes through similarly, noting that since $\mbox{deg}(E)\geq 4g+1$,  we have $h^0(C, E)-2>\frac{\mathrm{deg}(E)}{2}$.
\medskip

\emph{(ii)} Assume $E=N\oplus M$ splits and $h^0(N\otimes M^\vee )=h^0(M\otimes N^\vee)=0$. Put $k=\mathrm{min}\{\mathrm{deg}(N),\mathrm{deg}(M)\}$. Assume, for simplicity, $k=\mathrm{deg}(N)$. Since $E$ is globally generated and $h^1(C, E)=0$, it follows that both $N$ and $M$ are globally generated and non-special. Put $V_1:=H^0(C, N)$ and $V_2:=H^0(C, M)$. These two subspaces are isotropic and $H^0(C, E)=V_1\oplus V_2$. We prove that $\mathcal{R}_{\ge k}(C, E)= V_1^\vee\cup V_2^\vee$. Let $a\in \mathcal{R}_{\ge k}(C, E)\setminus \{0\}$, then there exists $b$ such that $0\ne a\wedge b\in K^\perp$ and hence $a$ and $b$ span a line bundle $L$ of degree at least $k$ inside $E=N\oplus M$.
It the induced map $L\to N$ is non-zero, then $L=N$ and, since $h^0(M\otimes N^\vee)=0$, it follows that the map $L\to M$ is zero, which implies $a\in V_1^\vee$. If the map $L\to N$ is zero, then $a\in V_2^\vee$.

\vskip 3pt

Conversely, assume we are given isotropic subspaces $V_1,V_2\subseteq H^0(C,E)$ such that $H^0(C,E)^\vee=V_1^\vee\oplus V_2^\vee$ and $\mathcal{R}_{\ge k}(C, E)= V_1^\vee\cup V_2^\vee$.
Let $N$ and $M$ be the line bundles of degree at least $k$ contained in $E$ generated by $V_1^\vee$ and $V_2^\vee$ respectively. By isotropy, it follows that $N$ and $M$ are globally generated with $V_1^\vee\subseteq H^0(C, N)$ and $V_2^\vee\subseteq H^0(C, M)$. Since $H^0(C, N)$ and $H^0(C, M)$ are also isotropic, and hence contained in the resonance, the are in fact contained in $\mathcal{R}_{\ge k}(C, E)$. The assumption  $\mathcal{R}_{\ge k}(C, E)= V_1^\vee\cup V_2^\vee$ implies $V_1^\vee = H^0(C, N)$ and $V_2^\vee = H^0(C, M)$.

\vskip 3pt

\emph{Claim 1.} The natural map $N\oplus M\to E$ is injective. Indeed, otherwise its image is a line bundle $L$ of degree at least $k$. Passing to global sections we obtain
\[
V_1^\vee\oplus V_2^\vee\to H^0(C, L)\subseteq H^0(C, E)
\]
and the composition is the identity. In conclusion the inclusion $L\subseteq E$ yields an equality $H^0(C, L)=H^0(C, E)$. Since $H^0(C, L)$ is isotropic, and hence contained in the resonance, we find $\mathcal{R}_{\ge k}(C, E)=H^0(C,E)^\vee$, in contradiction with the hypothesis.

\medskip

\emph{Claim 2.} $h^1(C, N)\cdot h^1(C, M)=0$. Assume on the contrary that both $N$ and $M$ are special. By Clifford's Theorem we obtain
\[
h^0(C, N)\le \frac{\mathrm{deg}(N)}{2}+1,\ \ h^0(C, M)\le \frac{\mathrm{deg}(M)}{2}+1
\]
and hence
\[
h^0(C, E)=h^0(C, N)+h^0(C, M)\le \frac{\mathrm{deg}(N)+\mathrm{deg}(M)}{2}+2.
\]
On the other hand, from \emph{Claim 1} we have $\mathrm{deg}(N)+\mathrm{deg}(M)\le \mathrm{deg}(E)$ which imples that
\[
h^0(C, E)\le \frac{\mathrm{deg}(E)}{2}+2.
\]
Since $h^1(C,E)=0$, by Riemann--Roch  $\mathrm{deg}(E)\le 4g$, contradicting the hypothesis.

\vskip 4pt

\emph{Claim 3.} Suppose $h^1(C, N)=0$. Then $L:=E/N$ is torsion-free. Indeed, if it has torsion, then we consider the line bundle $N':=\ker\{E\to L/\mathrm{tors}(L)\}\subseteq E$ which is also of degree at least $k$, and an inclusion
$N\subsetneq N'$.  Since $N$ is non-special, $N'$ is also non-special and by Riemann-Roch $H^0(C, N)\subsetneq H^0(C, N')$. Note however that $H^0(C, N')$ is isotropic, therefore contained in the resonance, contradicting that $H^0(C, N)$ is a component of $\mathcal{R}_{\ge k}(C,E)$.

\vskip 3pt

Having proved these claims, we conclude. Denoting by $\alpha\colon M\to L$ the composition,  note that $\alpha \ne 0$ for otherwise $M\subseteq N$, contradicting the hypothesis. Since $N$ is non--special, the equality $H^0(C,E)=H^0(C,N)\oplus H^0(C,M)$ and the long cohomology sequence of the exact sequence
\[
0\longrightarrow N\longrightarrow E\longrightarrow L\longrightarrow 0
\]
shows that $H^0(\alpha)\colon H^0(C, M)\to H^0(C, L)$ is an isomorphism. Since $E$ is globally generated, it follows that $L$ is globally generated as well. We have the following situation: $M$ and $L$ are globally generated line bundles, and $\alpha\colon M\to L$ is a morphism inducing an isomorphism on global sections. It implies that $\alpha$ is surjective, and hence it is an isomorphism, providing us with a splitting $E\cong N\oplus M$. To prove that $h^0(M\otimes N^\vee) = 0$ observe that any non-zero section in $H^0(M\otimes N^\vee)$ gives an embedding $N\subseteq N\oplus M$ with torsion--free quotient which yields to elements in $\mathcal{R}_{\ge k}(C, E)$ that are neither in $V_1^\vee$  nor in $V_2^\vee$.
\endproof


\begin{remark}
	For (ii), the bound $\mbox{deg}(E)\geq 4g+1$ in the assumption of Theorem \ref{thm:split} can be improved to $4g$ if $C$ is non-hyperelliptic. Indeed, in Claim 2, the inequalities resulting from Clifford's Theorem are strict.
\end{remark}

\begin{remark}
The resonance of split bundles is in general much more complicated than the union of two subspaces. The easiest example is obtained on the projective line for the bundle $\mathcal{O}(1)\oplus \mathcal{O}(1)$ whose resonance is a smooth quadric in the three--dimensional projective space. In this case, the stratification consists of only one stratum, the maximal one.

A more elaborate example is the following. Suppose $C$ is an elliptic curve, $p\ne q$ are two points on $C$, and $E=\mathcal{O}_C(3p)\oplus \mathcal{O}_C(3q)$. Then the projectivized resonance $\mathbf{R}(C,E)$ has three connected components, all of dimension two, namely, the planes $\mathbf{P}H^0(C,\mathcal{O}_C(3p))$, $\mathbf{P}H^0(C,\mathcal{O}_C(3q))$, and a ruled surface over the curve $\mathrm{Pic}^2(C)=C$. This description follows directly from \cite[Proposition 6.1]{AFRS}, by observing that for any $L\in \mathrm{Pic}^2(C)$ we have $h^0(C,E(-L))=1$. Note that, in this case, we have $$\mathbf{R}_{\ge 3}(C,E)=\mathbf{P}H^0(C,\mathcal{O}_C(3p))\cup\mathbf{P}H^0(C,\mathcal{O}_C(3q)),$$ which shows that the bound in the theorem above is not sharp.
\end{remark}

\end{document}